\documentclass[11pt]{article}
\usepackage{amsmath,amssymb,amsthm}
\usepackage{natbib}
\usepackage{graphicx}
\usepackage{titling}
\usepackage{xcolor}
\usepackage{pgfplots}
\usepackage{subcaption}
\usepackage{float}
\usepackage{nameref}
\usepackage{setspace}
\usepackage{titling}
\usepackage{bm}
\usepackage[skip=0.5pt]{caption}
\usepackage[font=small, labelfont=bf, textfont=it]{caption}
\usepackage[normalem]{ulem}
\usepackage[
    colorlinks=true,
    citecolor=blue,   
    linkcolor=blue,    
    urlcolor=blue     
]{hyperref}
\usepackage[
  a4paper,
  left=3cm,
  right=2.5cm,
  top=2.5cm,
  bottom=2.5cm
]{geometry}

\newcommand{\bX}{\boldsymbol{X}}

\newcommand{\PP}{\mathbb{P}}
\newcommand{\RR}{\mathbb{R}}
\newcommand{\EE}{\mathbb{E}}
\newcommand{\Xtil}{\tilde{X}}

\newcommand{\calH}{\mathcal{H}}
\newcommand{\calS}{\mathcal{S}}
\newcommand{\calN}{\mathcal{N}}

\newcommand{\FDP}{\text{FDP}}
\newcommand{\TPP}{\text{TPP}}
\newcommand{\fdp}{\text{fdp}}
\newcommand{\tpp}{\text{tpp}}
\newcommand{\FDR}{\text{FDR}}
\newcommand{\TPR}{\text{TPR}}

\newcommand{\TL}{\text{TL}}
\newcommand{\lasso}{\text{L}}
\newcommand{\cv}{\text{cv}}

\newcommand{\FPR}{\text{FPR}}

\newcommand{\lfdr}{\text{lfdr}}

\DeclareMathOperator*{\argmin}{arg\,min}

\newtheorem{theorem}{Theorem}[section]
\newtheorem{lemma}[theorem]{Lemma}
\newtheorem{proposition}[theorem]{Proposition}

\theoremstyle{definition}

\theoremstyle{remark}
\newtheorem{assumption}[]{Assumption}

\theoremstyle{remark}
\newtheorem{remark}[theorem]{Remark}

\newtheoremstyle{named}%
  {}{}            
  {\itshape}      
  {}              
  {\bfseries}     
  {.}             
  {.5em}          
  {\thmname{#1}\thmnote{ \textnormal{(#3)}}} 

\theoremstyle{named}
\newtheorem*{namedtheorem}{Theorem}

\title{Empirical Bayes Variable Selection with Lasso Statistics in the AMP Framework}

\author{
Lina Hidmi \qquad\qquad Asaf Weinstein\\[0.6em]
{\em Hebrew University of Jerusalem}\thanks{Emails: lina.hidmi@mail.huji.ac.il, asaf.weinstein@mail.huji.ac.il}\\[1em]
}

\date{}

\begin{document}

\maketitle

\begin{abstract}
The Lasso is one of the most ubiquitous methods for variable selection in high-dimensional linear regression and has been studied extensively under different regimes. In a particular asymptotic setup entailing $n/p\to \text{constant}$, an i.i.d.~Gaussian $X$ matrix and linear sparsity, \citet{su2017false} analyzed the Lasso selection path and presented negative results, showing that maintaining small levels of the false discovery proportion comes at a substantial cost in power. Followup work by \citet{wang2020bridge} used the same framework to study the tradeoff between type I error and power for thresholded-Lasso selection, which ranks the variables based on the magnitude of the Lasso estimate instead of the order of appearance on the Lasso path, and demonstrated that significant improvements are possible if the regularization parameter is chosen appropriately. 
We take this line of research a step further, seeking an {\em optimal} selection procedure in the AMP framework among procedures that order the variables by some univariate function of the Lasso estimate at a fixed value $\lambda$ of the regularization term. Observing that the model for the Lasso estimates effectively reduces asymptotically to a version of the well-studied two-groups model, we propose an empirical Bayes variable selection procedure based on an estimate of the local false discovery rate. We extend existing results in the AMP framework to obtain exact predictions for the curve describing the asymptotic tradeoff between type I error and power of this procedure. Additionally, we prove that the optimal $\lambda$ is the minimizer of the asymptotic mean squared error, and accordingly propose to use the empirical Bayes procedure with $\lambda$ estimated by cross-validation. The theoretical predictions imply that the gains in power can be substantial, and we confirm this by numerical studies under different settings.
\end{abstract}

\section{Introduction}
\label{sec:intro}

Consider the Gaussian linear regression model, 
\begin{equation} 
\label{eq:linear_model}
Y = \boldsymbol{X} \beta + \zeta, \quad \quad \zeta \sim \mathcal{N}_n(0, \sigma^2 \boldsymbol{I}), 
\end{equation}
where \( \boldsymbol{X} \in \mathbb{R}^{n \times p} \) is the observed matrix of predictor  measurements for \( p \) predictors on each of \( n \) subjects, and where the coefficient vector $\beta = (\beta_1, \beta_2, \dots, \beta_p)^\top$ and the noise level $\sigma^2$ are unknown. 
In modern applications such as genomics, image analysis, and text mining \citep{varoquaux2012small, yamashita2008sparse, gramfort2011generalized, beer2018incorporating} 
the number of available predictor variables $p$ is typically  large, sometimes larger than the sample size $n$, but it is often assumed {\em a priori} that the majority of them are  irrelevant, i.e.~that the corresponding $\beta_i$ are essentially zero. 
In these sparse settings a main objective is to perform variable selection, i.e.~to estimate the subset of predictors that truly affect the (conditional) distribution of the response. 

In the noisy model \eqref{eq:linear_model} no variable selection procedure is actually expected, for finite $n$ and $p$, to perfectly recover the subset of relevant variables, but the quality of a procedure can be assessed based on the degree to which it is able to select truly relevant variables, while avoiding selection of too many irrelevant ones. 
The {\em controlled} variable selection problem \citep{candes2018panning} formally treats variable selection as a multiple testing problem, specifically, that of simultaneously testing the $p$ null hypotheses 
\[
H_{i}: \beta_i = 0, \quad \quad i \in \calH \equiv \{1, 2, \dots, p\}, 
\]
so that selecting a variable amounts to rejecting the corresponding null. 
Let \( \calH_0 \equiv \{ i : \beta_i = 0 \} \) denote the (unknown) subset of true null hypotheses, and \( \calS \equiv \calH \setminus \calH_0 \) denote the subset of true non-null hypotheses. 
For any variable selection procedure that outputs an estimate \( \hat{\calS} \subseteq \{1, 2, \dots, p\} \) of \( \calS \), define the {\em false discovery proportion} (FDP) as the fraction of true nulls among the rejections, and the {\em true positive proportion} (TPP) as the fraction of rejections among the true non-nulls, 
\[
\text{FDP} = \frac{|\hat{\calS} \cap \calH_0|}{|\hat{\calS}|}, \quad \quad
\text{TPP} = \frac{|\hat{\calS} \cap \calS|}{|\calS|},
\]
where $| \mathcal{A}| $ is the cardinality of a set $\mathcal{A}$, and where we use the convention $0/0\equiv 0$. 
Taking expectations on these two random fractions under \eqref{eq:linear_model}, we obtain the {\em false discovery rate} (FDR) 
and the {\em true positive rate} (TPR), 
\[
\text{FDR} = \EE [\FDP], \quad \quad
\text{TPR} = \EE [\TPP],
\]
which are specific metrics for the type I error and, respectively, for the power of the procedure. 
An effective selection procedure achieves high TPR while keeping FDR small. 
Consistent with the classic multiple testing viewpoint, which treats type I and type II errors asymmetrically, 
in the controlled variable selection framework we will say that a procedure is {\em valid} at level $q\in [0,1]$ if it has $\FDR\leq q$ for all values of the parameter $\beta$, and the task would generally be to maximize (in some applicable sense) the TPR subject to validity. 

Many variable selection procedures can be represented as operating as follows. 
First, for each variable $i=1,...,p$, calculate an {\em importance statistic} $T_i$, with larger values presumably indicating more evidence against $H_i$.
The choice of the importance statistic now defines the {\em selection path} of the procedure, the set-valued mapping which associates each fixed threshold $t\in \RR$ with the subset
\begin{equation}
\label{eq:threshold-selection}
\hat{\calS}(t)\ := \ \{i\leq p: \; T_i> t\}. 
\end{equation}
We remark that, although the convention above will be used by default (unless noted otherwise), 
there is no essential distinction between procedures that select for large versus small values of the statistic $T_i$; thus, from here on, whenever referring to procedures of the form \eqref{eq:threshold-selection} we mean procedures that reject for either small $T_i$ or for large $T_i$, and we switch between the two conventions as convenient. 
When the threshold $t$ in \eqref{eq:threshold-selection} varies from large to small values, 
the selection path defines an increasing trajectory of candidate subsets of variables. 
Now that the  path is determined, the procedure  needs to decide where to stop including additional variables by specifying a data-dependent threshold $\hat{t}$, resulting in the ultimately selected subset 
$
\hat{\calS} = \hat{\calS}(\hat{t}) = \{i\leq p: T_i\ > \hat{t}\}. 
$
In this view of a selection procedure, the specification of $\hat{t}$ is responsible for {\em calibrating} the procedure for type I errors, whereas the choice of $T_i$ is responsible for the {\em ordering} of the variables, and therefore controls the power of the procedure. 
Much of the existing literature on controlled variable selection focuses on calibration, which is an important and challenging problem, especially if we require the generality described above. 
In fact, it is only in recent years that procedures such as knockoffs \citep{barber2015controlling, candes2018panning}, which wrap around an arbitrary (freely chosen) importance statistic to achieve finite sample FDR control, have emerged in the literature. 



Still, if we assume that a method for picking $\hat{t}$ such that $\FDR\leq q$ for an arbitrary choice of the importance statistic is available, then the entire procedure---in particular, its power at any target FDR level $q$---depends only on the choice of the statistic $T_i$. 
This motivates studying the selection path $\{\hat{\calS}(t),\ t\in \RR\}$ itself, setting aside the question of calibration. 
More precisely, if we define 
\begin{equation}
\label{eq:fdp-of-t}
    \text{FDP}(t) \equiv \frac{|\hat{\calS}(t) \cap \calH_0|}{|\hat{\calS}(t)|}, \quad \quad
\text{TPP}(t) \equiv \frac{|\hat{\calS}(t) \cap \calS|}{|\calS|}, 
\end{equation}
then we might consider the parametric curve  
$t \mapsto ( \EE[\TPP(t)], \, \EE[\FDP(t)])$, 
as a receiver operating characteristic (ROC) curve of the procedure (although in the standard definition of ROC the denominator in FDP is replaced by $|\calH_0|$).  
Since this parametric curve is determined solely by the choice of the importance statistic $T_i$, it is really the type I error-power tradeoff associated with the statistic $T_i$, and a natural question to ask is whether an {\em optimal} choice of the importance statistic exists in some sense. 

To address this question for fixed $n$ and $p$, i.e., to  compare different choices of the importance statistics for a fixed size of the problem, requires working with expectations, under \eqref{eq:linear_model}, of $\FDP(t)$ and $\TPP(t)$ for a given choice of $T_i$; this is prohibitive even for relatively simple choices of $T_i$ beyond plain least squares (which is anyway applicable only when $p<n$). 
For example, computation becomes infeasible already in the commonly encountered case where $T_i$ is a statistic based on the Lasso estimator \citep{tibshirani1996regression},
\begin{equation}
\label{eq:lasso-estimator}
\hat{\beta}(\lambda)
= \underset{b \in \mathbb{R}^p}{\operatorname{argmin}}
\frac{1}{2}\|Y - \boldsymbol{X} b\|_2^2 + \lambda \|b\|_1,
\end{equation}
since the distribution of $\hat{\beta} = \hat{\beta}(\lambda)$ is intractable for fixed $n$ and $p$.
Nevertheless, under a specific {\em asymptotic} setup, recent developments from approximate message-passing (AMP) theory \citep{bayati2011dynamics, bayati2012lasso} allow to obtain the {\em exact limits}  of $\FDP(t)$ and $\TPP(t)$ for statistics of the form 
\begin{equation}
\label{eq:T-h-lambda}
T_i = \varphi( \hat{\beta}_i(\lambda)),
\end{equation}
where $\varphi$ is a (sufficiently well-behaved) univariate   function. 
These remarkable advances were used and extended in subsequent work 
to analyze the (asymptotic) {\em tradeoff curve},  
\begin{equation*}
    t \ \mapsto  \big(  \lim_{p\to\infty} \TPP(t), \, \lim_{p\to\infty} \FDP(t)  \, \big), 
\end{equation*}
for different Lasso-based variable importance statistics. 
Specifically, in \cite{bogdan2013statistical} asymptotic predictions of FDP and TPP were first obtained for Lasso selection, i.e.~the procedure that selects the active coordinates of $\hat{\beta}$ in  \eqref{eq:lasso-estimator} and whose selection path is formed by varying $\lambda$. 
\cite{su2017false} later strengthened these results 
and derived a universal tradeoff diagram, applicable to any $\Pi$ and $\sigma^2$, which implies that simultaneously achieving high TPP and low FDP is impossible for Lasso selection. 
In \cite{wang2020bridge} the same asymptotic framework was used to study the tradeoff curve for selection based on thresholding the size of an $\ell_\gamma$-regularized estimate, which for $\gamma=1$ specializes to {\em thresholded}-Lasso selection, i.e.~the procedure that selects variable $i$ if $|\hat{\beta}_i|$ exceeds a threshold. 
Unlike Lasso selection, here $\lambda$ is regarded as fixed and the selection path is parametrized by the threshold. 
Among their contributions, \cite{wang2020bridge} characterize the optimal choice for $\lambda$ in terms of the tradeoff curve. 
Overall, their analysis shows that by using thresholded-Lasso with an appropriate choice of $\lambda$, considerable improvements are possible compared to Lasso selection. 

The aforementioned works build on an asymptotic simplification of the empirical distribution of the Lasso estimates, established in \citet{bayati2012lasso}, to analyze the tradeoff curves associated with two widely used procedures, namely Lasso selection and thresholded-Lasso selection. 
In the present work we take the opposite perspective, using the asymptotic simplification of the model to {\em guide} the choice of a selection procedure—or, equivalently, of the importance statistic $T_i$.
The main idea is to identify the simple model to which the Lasso estimator \eqref{eq:lasso-estimator} effectively reduces asymptotically under the AMP setup, as a version of the {\em two-groups} model  \citep{efron2001empirical, efron2008microarrays}. 
Drawing on the extensive literature on multiple testing under the two-groups model \citep{sun2007oracle, xie2011optimal, heller2021optimal}, we propose to use the corresponding {\em local false discovery rate} (lfdr) of the  coefficient estimate to order the variables, so that smaller lfdr indicates stronger evidence against the null. 
In term of the selection path this is equivalent  to rejecting for small values of 
\begin{equation}
\label{T_ol-intro}
T_i^* = \frac{f_0(\hat{\beta}_i)}{f(\hat{\beta}_i)}    
\end{equation}
where $f_0$ denotes the null density and $f$ is the marginal (mixture) density under the corresponding two-groups model. 
However, $T_i^*$ depends on unknown parameters and is therefore not a legal but an {\em oracle} choice of the statistic, which we indicate by the superscript asterisk. 
In fact, unlike the standard two-groups model, where $f$ is unknown but $f_0$ is assumed known, in our case also $f_0$ is known only up to a set of parameters. 
Taking an empirical Bayes approach, we set out to mimic the aforementioned oracle procedure by plugging in estimates of $f$ and $f_0$ into \eqref{T_ol-intro}. 
For the mixture density $f$ we use a kernel density estimator, whereas the null density $f_0$, which based on existing asymptotic theory admits a known parametric form, is estimated parametrically. 

From a theoretical standpoint, obtaining asymptotic predictions for the proposed EB procedure is 
considerably more challenging compared to the analysis for Lasso or thresholded-Lasso, because the fact that $\hat{f}$ and $\hat{f}_0$ depend on all $\hat{\beta}_1(\lambda),...,\hat{\beta}_p(\lambda)$ prevents us from directly applying the fundamental result from  \citet{bayati2012lasso} to the EB {estimate} of $T^*_i$. 
Still, intuitively we expect that if the estimators of $f$ and $f_0$ are consistent---which, in our approach, requires among other conditions that the kernel density estimator be consistent---then the limiting FDP and TPP of the EB procedure will coincide with those of the oracle. 
The main theoretical result in this paper says that under suitable conditions this is indeed the case, our analysis essentially disentangling the limit of $\widehat{\lfdr}(t)$ for a fixed $t$ from the limit of the average of $\lfdr(\hat{\beta}_j)$, the true lfdr at the coefficient estimates, over the $p$ coefficient estimates. 
%
Under suitable assumptions, for any fixed $\lambda$ our EB procedure is shown to asymptotically attain the optimal tradeoff curve among all functions $\varphi$ in \eqref{eq:T-h-lambda}; 
this is a strong notion of asymptotic {\em instance}-optimality, familiar from the EB literature and tracing back to Robbins's original ideas \citep{robbins1951asymptotically}.


\begin{figure}[]
\centering
\includegraphics[width=0.6\linewidth]{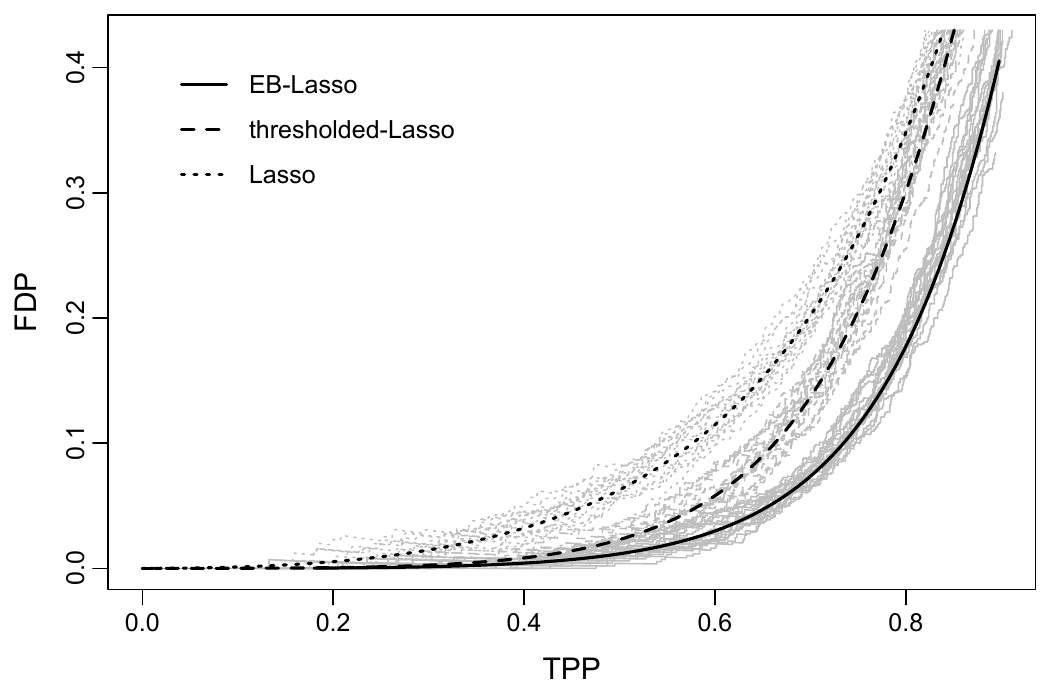}
\caption{Tradeoff curves for the proposed EB method (``EB-Lasso" in the legend) compared to Lasso and thresholded-Lasso selection, for an example with $p=10^4, n=2p$, and  $\beta_i\sim 0.9\delta_0 + 0.1\calN(3.5, 1)$. 
Black curves are theoretical asymptotic predictions, thin grey lines are corresponding simulation results. 
}
\label{fig:intro}
\end{figure}

Similarly to thresholded-Lasso, the choice of $\lambda$ can have a substantial effect on the tradeoff curve of the proposed EB procedure. 
For thresholded-Lasso (and, more generally, for thresholded bridge selection), \cite{wang2020bridge} proved that the value of $\lambda$ which minimizes the asymptotic mean squared error (MSE) is also asymptotically optimal in terms of the FDP-TPP tradeoff. 
In Theorem \ref{thm:optimal-lambda} we obtain an analogous result for the proposed EB procedure, proving that the minimizer in $\lambda$ of the asymptotic MSE again yields the optimal tradeoff curve; although the conclusion agrees with that of \cite{wang2020bridge} for thresholded-bridge selection, our proof requires different techniques. 
This result motivates us to use cross validation for estimating the optimal $\lambda$, which depends on unknown parameters, and from results in \cite{weinstein2023power} we deduce the limit of a $K$-fold
cross-validation estimator under the AMP framework. 
Altogether, our results suggest that, up to the slight bias due to cross validation, using the EB selection procedure with the proposed estimator of $\lambda$ will have the optimal asymptotic tradeoff curve among the larger class of procedures obtained from \eqref{eq:T-h-lambda} by allowing to search not only over $\varphi$ but also over $\lambda$. 



%
To demonstrate the advantages of the proposed method, 
Figure \ref{fig:intro} shows results for an example  with $p=10^4, n=2\cdot 10^4, \sigma = 1$, and where $\beta_i$ are i.i.d.~from a mixture of a point mass at zero with probability $1-\epsilon=0.9$ and a normal distribution $\mathcal{N}(3.5,1)$ with probability $\epsilon = 0.1$. 
The figure shows the corresponding asymptotic FDP–TPP tradeoff curve, as predicted by our theory, of the proposed EB procedure, along with the asymptotic predictions, from existing work, of Lasso and thresholded-Lasso selection; here $\lambda=1$ was used to calculate the Lasso estimate \eqref{eq:lasso-estimator} for EB and thresholded-Lasso. 
For each of the three methods we overlaid empirical (realized) tradeoff curves from 17 independent simulation runs, which show good agreement with the theory. 
This figure demonstrates that the improvement in power can be substantial; further examples, using different signal distributions and sampling ratios, are presented later on. 

The rest of the paper is organized as follows. 
In Section \ref{sec:AMP-framework} we recall the AMP setup and some relevant existing results. 
Section \ref{sec:eb-selection} introduces an oracle selection procedure and the EB procedure designed to asymptotically compete with it. 
Theoretical analysis of both the oracle and the EB procedures is presented in Section \ref{sec:theoretical}, which also includes results on the optimal choice of $\lambda$. 
Section \ref{sec:numerical} reports results from numerical studies where we assess the agreement with the theory. 
In Section \ref{sec:local-fdr} we adopt a viewpoint from recent work emphasizing 
the role of the lfdr in defining {\em type I error} rates of multiple testing problems; 
it is demonstrated how predictions of the lfdr in the AMP framework can be used in a more exploratory manner to select variables while maintaining asymptotic control of the lfdr (instead of FDR). 
Section \ref{sec:discussion} concludes with a discussion of some issues not addressed in this paper and which merit further investigation.

\section{The AMP setup and review of existing results}
\label{sec:AMP-framework}
 We work under the {approximate message passing} (AMP) framework of \cite{bayati2012lasso}, employed also in \cite{su2017false} and \cite{wang2020bridge}. 
Specifically, assuming the linear model \eqref{eq:linear_model}, we consider a high-dimensional asymptotic setting where $p, n \to \infty$ such that $n/p \to \delta > 0$, and the entries of the $n \times p$ matrix $\boldsymbol{X}$ are i.i.d.~realizations from a $\mathcal{N}(0, 1/n)$ distribution. 
Furthermore, individual coefficients $\beta_i$ are assumed to be i.i.d.~copies of a random variable $\Pi$ with bounded second moment and a point mass at zero, 
\begin{equation}
\label{eq:iid-two-groups}
    \beta_i \ \overset{iid}{\sim} \ \Pi = (1-\epsilon)\delta_0 + \epsilon \Pi_1, \ \ \ \ \ \ \ \ \ i=1,...,p, 
\end{equation}
independent of $\boldsymbol{X}$ and $\zeta$, where 
$\mathbb{P}(\Pi_1 =0) = 0$. 
The constant $\epsilon \in (0,1)$ and the distribution of $\Pi_1$ are assumed unknown. 
In the sequel, we use $\Pi$ and $\Pi_1$ to denote either the random variable or its distribution, when there is no risk of confusion. 
We remark that the results to be presented apply also in a frequentist setup as long as the empirical distribution $\Pi^{(n)}$ of the entries of the fixed vector $\beta^{(n)} = (\beta^{(n)}_1,...,\beta^{(n)}_p)$ converges weakly 
to some distribution $\Pi$ of the form \eqref{eq:iid-two-groups} with bounded second moment. 

Under this framework, \citet{bayati2012lasso} introduced the following result, which characterizes the asymptotic behavior for the empirical distribution of the Lasso estimates. 
This result is central to our work, both motivating the new variable selection procedure introduced later, and serving as a key tool in the theoretical analysis.

\begin{namedtheorem}[Theorem 1.5 of \citet{bayati2012lasso}]\label{thm:manual}
    Under the asymptotic setup described above, let $\hat{\beta} = \hat{\beta}(\lambda)$ denote the Lasso estimator \eqref{eq:lasso-estimator}. Let $\psi: \mathbb{R} \times \mathbb{R} \to \mathbb{R}$ be a pseudo-Lipschitz function. Then
\[
\frac{1}{p} \sum_{i=1}^{p} \psi(\hat{\beta}_i, \beta_i) \xrightarrow{\text{a.s.}} \mathbb{E}\left[ \psi\left( \eta_{\alpha \tau}(\Pi + \tau Z), \Pi \right) \right]
\]
as $p\to \infty$, 
where $Z \sim \mathcal{N}(0,1)$ is independent of $\Pi$, and the parameters $\tau > 0$ and $\alpha > \alpha_{\min}$ are the unique solution to the  system of equations given by
\begin{align}
\tau^2 &= \sigma ^2 + \frac{1}{\delta} \mathbb{E}\left[ \left( \eta_{\alpha \tau}(\Pi + \tau Z) - \Pi \right)^2 \right] \label{eq:tau} \\[1ex]
\lambda &= \alpha \tau \left( 1 - \frac{1}{\delta} \mathbb{P}(|\Pi + \tau Z| > \alpha \tau) \right).
\label{eq:lambda}
\end{align}
Here the function $\eta_t(x) = \mathrm{sign}(x) \cdot (|x| - t)_+$ is the soft-thresholding operator, and $\alpha_{\min}$ is the unique non-negative solution to the equation
\begin{equation}\label{eq:alphamin}
(\alpha^2 + 1)\Phi(-\alpha) - \alpha \phi(\alpha) = \frac{\delta}{2},
\end{equation}
where $\Phi$ and $\phi$ denote the c.d.f.~and p.d.f.~of the standard normal distribution, respectively.
\end{namedtheorem}
The theorem above basically says that, as far as calculating expectations of sufficiently well-behaved functions under the (random) empirical distribution of the pairs $(\hat{\beta}_i, \beta_i)$, we can asymptotically regard 
\begin{equation}
\label{eq:asymptotic-iid}
\big( \hat{\beta}_i, \beta_i) \ \sim \ (\eta_{\alpha\tau}( \Pi +\tau Z), \; \Pi \big), \ \ \ \ \ \ \ \text{independent for}\ i=1,...,p. 
\end{equation}
In \citet{bogdan2013statistical} this fundamental result of \citet{bayati2012lasso} was leveraged, with some necessary technical extensions, to obtain exact limits of FDP and TPP for 
the procedure whose path is given by 
\begin{equation}
\label{eq:lasso-selection}
\hat{\calS}^L(\lambda) = \{i: \hat{\beta}_i(\lambda) \neq 0 \},\ \ \ \ \ \lambda>0. 
\end{equation}
Formally, putting $\FDP(\lambda):= |\hat{\calS}^L(\lambda) \cap \calH_0|/|\hat{\calS}^L(\lambda)|$ and $\TPP(\lambda):= |\hat{\calS}^L(\lambda) \cap \calS|/|\calS|$, they proved that under the AMP setup described above, 
\begin{equation}\label{eq:: FDP TPP lasso}
\begin{aligned}
&\text{FDP}(\lambda) \xrightarrow{P} \frac{2(1 - \epsilon)\Phi(-\alpha)}{2(1 - \epsilon)\Phi(-\alpha) + \epsilon \mathbb{P}(|\Pi_1 + \tau Z| > \alpha \tau)} \ =: \ \fdp^\lasso(\lambda)\\[1.5ex]
&\text{TPP}(\lambda) \xrightarrow{P} \mathbb{P}(|\Pi_1 + \tau Z| > \alpha \tau) \ =: \ \tpp^\lasso(\lambda), 
\end{aligned}
\end{equation}
as $p\to \infty$, where $(\alpha, \tau)$ solve equations \eqref{eq:tau} and \eqref{eq:lambda}. 
\cite{su2017false} later extended this to show that the convergence in \eqref{eq:: FDP TPP lasso} is uniform in $\lambda$. 
It is noted that the procedure \eqref{eq:lasso-selection} is, strictly speaking, not of the form \eqref{eq:threshold-selection}; however, as pointed out in \cite{weinstein2017power}, it is closely related to the procedure of the form \eqref{eq:threshold-selection} that uses the time of entry into the Lasso path  
$T^{\text{LM}}_i = \sup\{\lambda>0:\hat{\beta}_i(\lambda) \neq 0\}$ 
as importance statistic. 
Indeed, the two procedures have different selection paths only if variables drop out of the Lasso path, which we consider to be a fairly rare event. 
Thus, for `practical' purposes, we regard  \eqref{eq:lasso-selection} as equivalent to the path 
\begin{equation}
\label{eq:lambda-sup-path}
\hat{\calS}^{\text{LM}}(t) = \{i\leq p: T^{\text{LM}}_i \geq t\},\ \ \ \ \ t>0, 
\end{equation}
where the superscript stands for ``Lasso-max", 
and we refer to \eqref{eq:lambda-sup-path} as the path corresponding to {\em Lasso} selection. 


In \citet{su2017false} the asymptotic predictions \eqref{eq:: FDP TPP lasso} were used to demonstrate a deficiency of Lasso selection, showing that under the AMP setting, regardless of $\Pi_1$ and $\epsilon$, high TPP necessarily comes at the price of inflated FDP. 
Followup work under the AMP framework \citep{wang2020bridge, weinstein2023power} has shown that this can be mitigated considerably by selecting only the Lasso estimates whose magnitude $|\hat{\beta}_i(\lambda)|$ is sufficiently large, that is, by employing {\em thresholded-Lasso} selection. 
For fixed $\lambda>0$, the corresponding selection path is
\begin{equation}
\hat{\calS}^{\TL}_\lambda(t) = \{i: |\hat{\beta}_i(\lambda)| > t \},\ \ \ \ \ t>0. 
\label{eq:thresholded-lasso}
\end{equation}
Defining $\FDP(t;\lambda):= |\hat{\calS}^{\TL}_\lambda(t) \cap \calH_0|/|\hat{\calS}^{\TL}_\lambda(t)|$, $\TPP(t;\lambda):= |\hat{\calS}^{\TL}_\lambda(t) \cap \calS|/|\calS|$,  \citet{wang2020bridge} prove that 
\begin{equation}
\label{eq:FDP TPP thresholded-lasso}
\begin{aligned}
&\text{FDP}(t; \lambda) \xrightarrow{P}
\frac{2(1 - \epsilon)\Phi\left(-\alpha - \frac{t}{\tau}\right)}
     {2(1 - \epsilon)\Phi\left(-\alpha - \frac{t}{\tau}\right)
     + \epsilon \, \mathbb{P}\left( \left| \Pi_1 + \tau Z \right| \geq \alpha \tau + t \right)} \ =: \ \fdp^{\TL}(t; \lambda)\\[1ex]
&\text{TPP}(t; \lambda) \xrightarrow{P}
\mathbb{P}\left( \left| \Pi_1 + \tau Z \right| \geq \alpha \tau + t \right) \ =: \ \tpp^{\TL}(t; \lambda). 
\end{aligned}
\end{equation}
The asymptotic predictions \eqref{eq:FDP TPP thresholded-lasso} were used in \citet{wang2020bridge} and \cite{weinstein2023power}  to show that if $\lambda$ is chosen adequately, this strategy---which allows looking further down the Lasso path by choosing smaller values of $\lambda$---can achieve significantly better separation between signal and noise compared to Lasso selection. 
Formally, \citet[][Theorem 3.2]{wang2020bridge} prove the following. 
Let $\lambda^*$ be the minimizer of the asymptotic mean squared error $\EE (\eta_{\alpha\tau}(\Pi + \tau Z)-\Pi)^2$. 
Then, first, $\lambda^*$ is optimal for thresholded-Lasso in the sense that it minimizes the asymptotic FDP for any fixed asymptotic TPP level. Second, they show that for any (asymptotic) TPP level feasible for Lasso, it holds that
\begin{equation}
\tpp^\TL(t; \lambda^*) = \tpp^\lasso(\lambda) \ \ \Rightarrow \ \ 
\fdp^\TL(t; \lambda^*) \leq \fdp^\lasso(\lambda).     
\end{equation}
In other words, if $\lambda$ for Lasso and $t = t(\lambda^*)$ for thresholded-Lasso are chosen such that both procedures have the same TPP, then thresholded-Lasso has smaller FDP than Lasso.

%


\section{An empirical Bayes variable selection procedure}
\label{sec:eb-selection}

Returning to the interpretation presented above for Theorem 1.5 of \citet{bayati2012lasso}, notice that \eqref{eq:asymptotic-iid} entails a {\em two-groups} model, 
\begin{equation}
\label{eq:two-groups}
\begin{aligned}
    \mathbb{P}(\beta_i = 0)  &= 1-\epsilon, &\quad \hat{\beta_i}|\beta_i = 0\  &\sim \ \eta_{\alpha\tau}(\tau Z) \\[1ex]
    \mathbb{P}(\beta_i \neq 0)  &= \epsilon, &\quad \hat{\beta_i}|\beta_i \neq 0\  &\sim \ \eta_{\alpha\tau}( \Pi_1 +\tau Z). 
\end{aligned}
\end{equation}
This observation motivates us to appeal to existing theory in the literature on multiple testing under the two-groups model, and propose an EB approach that uses an estimate of the {local false discovery rate} as the test statistic. 
Thus, compared to {thresholded-Lasso}, which  selects  variables by thresholding the magnitude of the Lasso coefficient, we will propose to select them using an EB procedure that thresholds the estimated local false discovery rate of $\hat{\beta_i}$ under \eqref{eq:two-groups}. 
In this section, we develop the EB procedure under the {\em working assumption} that $(\hat{\beta}_i, \beta_i)$ are i.i.d.~pairs from the two-groups model \eqref{eq:two-groups}; to clarify, postulating the i.i.d.~model in \eqref{eq:two-groups} is completely acceptable for the purpose of developing the methodology itself, whereas the analysis in the next section will of course need to rely  only on the much weaker property in the result of \citet{bayati2012lasso}. 
To expand on our motivation, in the genuinely i.i.d.~setup, the EB selection procedure is expected to yield a selection procedure with asymptotic {\em optimality} properties in terms of power. 
More precisely, the oracle selection rule that thresholds the {true} lfdr can be shown to asymptotically maximize the expected TPP  subject to FDR control. In fact, 
%
this follows from the {\em non}-asymptotic results of \citet{sun2007oracle} for mFDR control, or of \citet{heller2021optimal} for FDR control 
 \citep[see also][for analogous nonasymptotic results characterizing the oracle rule when $\beta_i$ are fixed parameters]{weinstein2021permutation}, because mFDR and FDR are asymptotically equivalent when the threshold is fixed \citep{benjamini2008microarrays}. 
In turn, the EB rule is designed to mimic the oracle procedure without using any knowledge of the parameters of the problem. 
Roughly speaking, then, under conditions ensuring consistency of the EB estimator  of the lfdr, it is expected to asymptotically attain the performance of the oracle. 

%

\smallskip
Proceeding under our working assumption, denote
$$
f_0 = \text{density of } \eta_{\alpha\tau}(\tau Z), \quad \quad \quad f_1 = \text{density of } \eta_{\alpha\tau}( \Pi_1 +\tau Z),
$$
where `density' should be understood here in the extended sense, since both $f_0$ and $f_1$ have mass at zero. 
As usual, we will refer to $f_0$ as the {\em null} distribution and to $f_1$ as the {\em non-null} distribution. 
Then the marginal density of $\hat{\beta_i}$ is the mixture
\begin{equation}\label{eq:f}
    f = (1 - \epsilon) f_0 + \epsilon f_1. 
\end{equation}
The {\em local false discovery rate} \citep[lfdr,  henceforth;][]{efron2004large, efron2008microarrays} is defined as the posterior probability, under the postulated model \eqref{eq:two-groups}, that ${\beta _i}$ 
is null given $\hat{\beta}_i=x$, 
\begin{equation}
\label{lfdr}
\text{lfdr}(x) := \PP(\Pi = 0 | \, \eta_{\alpha\tau}( \Pi +\tau Z) =x ) = \frac{(1 - \epsilon) f_0(x)}{f(x)}. 
\end{equation}
Thus, small values of $\text{lfdr}(\hat{\beta}_i)$ indicate more evidence against the null, $\beta_i = 0$. 
Note that, as in the standard two-groups model, $f_1$ and $\epsilon$ are unknown, but unlike in the usual two-groups  model, here also $f_0$ is unknown because $\alpha$ and $\tau$ depend on the unknown $\epsilon$ and $\Pi_1$ (and $\sigma^2$). 
Now consider the {\em oracle} procedure that selects variables with sufficiently small lfdr \eqref{lfdr}. 
This has the same selection path as the procedure that rejects for small values of the density ratio $T_i^*$ defined in \eqref{T_ol-intro}, 
which will be easier for us to work with 
as it does not require estimating $\epsilon$. 
%
Throughout we impose the condition that variables with $\hat{\beta}_i=0$ are not selected; this 
is innocuous because, assuming $\epsilon$ is small, the FDP of a rule that selects the zero coefficient estimates is usually too large to be of interest anyway. 
Thus, the selection path of the oracle procedure is given by 
\begin{equation*}
\hat{\calS}^{*}(t) 
= \bigg\{ i : \hat{\beta}_i \neq 0, \ \frac{f_0(\hat{\beta}_i)}{f(\hat{\beta}_i)} \leq t \bigg\}
\end{equation*}
and the corresponding tradeoff curve is given by 
\[
\text{FDP}^{*}(t) = \frac{| \hat{\calS}^{*}(t) \cap \calH_0 |}{| \hat{\calS}^{*}(t) |},\ \ \ \ \ \ \ 
\text{TPP}^{*}(t) = \frac{| \hat{\calS}^{*}(t) \cap {\calS}|}{| {\calS} |}. 
\]
To mimic the oracle rule with a legitimate selection rule, i.e.~a procedure that does not depend on unknown parameters of the problem, we propose to estimate $f$ (non-parametrically) and $f_0$ (parametrically) from the vector of observations $\hat{\beta}$, and substitute them in \eqref{T_ol-intro} to obtain an estimate 
$$
T^{EB}_i= \frac{\hat{f}_0(\hat{\beta}_i)}{\hat{f} (\hat{\beta}_i)}
$$
of $T^*_i$. 
The selection path of the resulting {\em empirical Bayes} (EB) procedure is then  
\begin{equation*}
\hat{\calS}^{EB}(t) = \left\{ i : \hat{\beta}_i \neq 0, \ T^{EB}_i \leq t \right\}
= \bigg\{ i : \hat{\beta}_i \neq 0, \ \frac{\hat{f}_0(\hat{\beta}_i)}{\hat{f}(\hat{\beta}_i)} \leq t \bigg\}, 
\end{equation*} 
and its tradeoff curve is given by 
\[
\text{FDP}^{EB}(t) = \frac{| \hat{\calS}^{EB}(t) \cap \calH_0 |}{| \hat{\calS}^{EB}(t) |},\ \ \ \ \ \ \ 
\text{TPP}^{EB}(t) = \frac{| \hat{\calS}^{EB}(t) \cap {\calS}|}{| {\calS} |}. 
\]


Before describing the specific implementation we propose for the EB procedure, we offer yet another representation for the selection paths of $\hat{\calS}^*$ and $\hat{\calS}^{EB}$, which will be useful for the analysis in the next section. 
Under the two-groups model \eqref{eq:two-groups}, both $f_0$ and $ f_1$ have point mass at zero, so we can write   
\begin{equation}
    f_0(\hat{\beta}_i) = (1 - w_0) \delta_0 (\hat{\beta}_i) + w_0 g_0(\hat{\beta}_i),  \label{f_0}
\end{equation}
where 
$$
1-w_0 =  \PP(|Z|  \leq \alpha) = 
2\Phi(\alpha) -1, \quad\quad
g_0 = \text{density of } \tau\eta_{\alpha}(Z)  \big\lvert |Z| >\alpha. 
$$ 
Similarly, 
\begin{equation}
\label{f_1}
    f_1(\hat{\beta}_i) = (1 - w_1) \delta_0 (\hat{\beta}_i) + w_1 g_1(\hat{\beta}_i),  
\end{equation}
where 
$$
\begin{aligned}
&1-w_1 = 
\mathbb{P}(|\Pi_1 + \tau Z| \leq \alpha\tau), \quad\quad\quad 
g_1 = \text{density of $\eta_{\alpha\tau}(\Pi_1 + \tau Z) \big\lvert |\Pi_1 + \tau Z|>\alpha\tau$}. 
\end{aligned}
$$
With these definitions, the marginal density of $\hat{\beta_i}$ in \eqref{eq:f} can be represented 
\begin{equation} \label{eq:f-alt}
    f(\hat{\beta}_i) = (1 - w) \delta_0 (\hat{\beta}_i) + w g(\hat{\beta}_i),
\end{equation}
where 
$$
1-w = (1-w_0) (1-\epsilon) + (1-w_1) \epsilon
$$
and 
$$
g = \frac{2(1-\epsilon)(1-\Phi(\alpha))}{w} g_0 + \bigg(1-\frac{2(1-\epsilon)(1-\Phi(\alpha))}{w} \bigg) g_1. 
$$
If we further define $q(\hat{\beta}_i) :=  w g(\hat{\beta}_i)$ and $q_0(\hat{\beta}_i) :=  w_0 g_0(\hat{\beta}_i)$, then \eqref{f_0} and \eqref{eq:f-alt} become 
\begin{equation}
    f_0(\hat{\beta}_i) = (1 - w_0) \delta_0 (\hat{\beta}_i) +  q_0(\hat{\beta}_i), \quad\quad\quad f(\hat{\beta}_i) = (1 - w) \delta_0 (\hat{\beta}_i) +  q(\hat{\beta}_i), 
\end{equation}
respectively. 
We then propose to estimate $f$ as
\begin{equation}
\label{f_hat}
\hat{f}(\hat{\beta}_i) = (1 - \hat{w}) \delta_0 (\hat{\beta}_i) + \hat{q}(\hat{\beta}_i), 
\end{equation}
with 
\begin{equation}
\label{eq:w-hat-q-hat}
\hat{w} = \frac{1}{p}\sum_{i=1}^{p}\mathbf{1}\{\hat{\beta_i} \neq 0\}, \quad\quad\quad\quad \hat{q}(x) = \frac{1}{ph} \sum_{i=1}^{p}\mathbf{1}\{\hat{\beta_i} \neq 0\} \ \mathcal{K}\bigg(\frac{x-\hat{\beta_i}}{h}\bigg), 
\end{equation}
throughout, we used the standard Gaussian density $\phi$ as the kernel, i.e, $\mathcal{K} = \phi$. 
 
For estimating $f_0$ 
we need to estimate the parameters $\alpha \tau$ and $\tau$, since $q_0$ involves the density $g_0$ of $\eta_{\alpha \tau}(\tau Z)=\tau\eta_{\alpha}(Z)$ conditional on $ |Z|>\alpha$. 
In \citet{mousavi2018consistent} it is shown that, for any fixed $\lambda > 0$,  
\begin{equation}
\label{tau-alphatau-hat}
\hat{\tau}^2 = \frac{\| Y - \bX \hat{\beta} \|_2^2}{n (1 - \| \hat{\beta} \|_0/n )^2}, 
\quad \quad \quad
\widehat{\alpha \tau} = \frac{\lambda}{1 - \| \hat{\beta} \|_0/n},\ \ \ \ \ 
\end{equation}
are consistent estimators of $\tau^2 $ and $\alpha\tau$, respectively. 
This naturally leads to the plug-in estimate 
$$
\hat{g}_0(x) = I(x>0)\displaystyle \frac{1}{2 \hat{\tau} (1-\Phi(\hat{\alpha}))} \phi\left( \frac{x + \widehat{\alpha \tau}}{\hat{\tau}} \right) + I(x<0) \ \displaystyle \frac{1}{2 \hat{\tau} (1-\Phi(\hat{\alpha}))} \phi\left(  \frac{x - \widehat{\alpha \tau}}{\hat{\tau}} \right)
$$
for $\hat{\tau} = \sqrt{\hat{\tau}^2}$ and $\hat{\alpha} = \widehat{\alpha \tau}/\hat{\tau}$. 
Further defining $\hat{w}_0 = 2(1-\Phi(\hat{\alpha}))$, this yields 
\begin{equation}
\hat{f}_0(\hat{\beta}_i) = (1 - \hat{w}_0) \delta_0 (\hat{\beta}_i) + \hat{q}_0(\hat{\beta}_i)
\label{f_0_hat} 
\end{equation}
as an estimator of $f_0$. 
Finally, because for $\hat{\beta}_i\neq 0$ we have $f_0(\hat{\beta}_i)/f(\hat{\beta}_i) = q_0(\hat{\beta}_i)/q(\hat{\beta}_i)$ and $\hat{f}_0(\hat{\beta}_i)/\hat{f}(\hat{\beta}_i) = \hat{q}_0(\hat{\beta}_i)/  \hat{q}(\hat{\beta}_i)$, 
the selection path of the oracle can be written equivalently as
\begin{equation}
\label{eq:s_hat_ol}
\hat{\calS}^{*}(t) = 
\bigg\{ i : \hat{\beta}_i \neq 0, \frac{ {q}_0(\hat{\beta}_i)}{{q}(\hat{\beta}_i)} \leq t \bigg\}, 
\end{equation}
and the selection path of the EB procedure can be written as
\begin{equation}
\label{eq:s_hat_eb}
\hat{\calS}^{EB}(t) = \bigg\{ i : \hat{\beta}_i \neq 0,\frac{\hat{q}_0(\hat{\beta}_i)}{\hat{q}(\hat{\beta}_i)} \leq t \bigg\}. 
\end{equation}


\section{Theoretical results} \label{sec:theoretical}

In the previous section we introduced an oracle procedure that ranks the variables by their local false discovery rate, as well as a legal, empirical Bayes procedure designed to mimic the oracle. 
In this section we analyze these two selection procedures under the AMP framework. 
Ultimately, the purpose is to obtain asymptotic predictions of the FDP and the TPP for both procedures, so their tradeoff curve can be compared to each other and to competitors such as Lasso selection and thresholded-Lasso selection. 
The analysis is technically challenging because instead of assuming that $(\hat{\beta}_i, \beta_i)$ are i.i.d.~draws from \eqref{eq:two-groups}, as in the two-groups setup studied in most of the existing work, here we have to settle for the considerably weaker property in Theorem 1.5 of Bayati and Montanari. 
This becomes an issue both in the analysis of the limit of the lfdr estimate itself, 
and mainly when calculating asymptotic predictions for the resulting EB procedure, because the fact that the lfdr is estimated from all $\hat{\beta}_i$ prevents us from directly applying the theorem of Bayati and Montanari to calculate the limits of FDP and TPP. 


Treating $\lambda$ as fixed, we will make the following mild assumption on the model, which says that the set of all $x$ for which
$f_0(x)/f(x) \leq t$
can be expressed as a union of disjoint intervals, excluding some  neighborhood of zero. 
This condition enables us to leverage the results of \cite{bogdan2013supplement} to prove  convergence of FDP and TPP for the oracle procedure to their limits; 
more specifically, it allows us to invoke the uniform continuity of the densities of the nonzero Lasso estimates on $(-\infty, -\Delta]$ and on $[\Delta, \infty), \Delta >0$. 



\begin{assumption}
\label{assmptn:1}
There exists $t_0 >0$ such that for all $t \leq t_0$, we have
$$
\frac{f_0(x)}{f(x)} \leq t \iff  x \in \bigcup_{m=1}^{M} (a_m(t), b_m(t)), 
$$
where $(a_m(t), b_m(t)),\ m=1,...,M$, are disjoint intervals (allowing $a_m(t) = -\infty$ or $b_m(t)=\infty$). 
Moreover, each interval is separated from zero, that is, for every $m$ 
\[
a_m(t), b_m(t) > 0 \text{ and } a_m(t) > \Delta \quad \quad \text{or} \quad \quad a_m(t), b_m(t) < 0 \text{ and } b_m(t) < -\Delta
\] 
holds 
for some small fixed $\Delta > 0$.  
\end{assumption}

Assumption \ref{assmptn:1} guarantees that
\[
\big\{x:\frac{f_0(x)}{f(x)} \leq t\big\} = \big\{x: x \neq 0, \frac{f_0(x)}{f(x)} \leq t\big\} = \big\{x:\frac{q_0(x)}{q(x)} \leq t\big\}
\]
%
%
For illustration, if $\Pi = 0.9  \delta_0 + 0.02  \,  \mathcal{N}(-3.6, 1) +\ 0.08  \,  \mathcal{N}(4, 1)$, then, e.g., ${f_0(x)}/{f(x)} \leq 0.6$  whenever $x \geq 1.413\ $ or $x \leq -1.941$. 


\smallskip
We begin with stating theoretical results for the oracle procedure. 
Fix $t\leq t_0$, where $t_0$ is the constant from Assumption \ref{assmptn:1}, and consider selecting the subset \eqref{eq:s_hat_ol}. 
Let 
$$
\begin{aligned}
V^*(t;\lambda) &= 
\sum_{i=1}^p \mathbf{1}\{\hat{\beta}_i \neq 0, \frac{ {q}_0(\hat{\beta}_i)}{{q}(\hat{\beta}_i)} \leq t, \beta_i = 0\}=
\sum_{i=1}^p \mathbf{1}\{\frac{ {q}_0(\hat{\beta}_i)}{{q}(\hat{\beta}_i)} \leq t, \beta_i = 0\}
, \\
R^*(t;\lambda) &= 
\sum_{i=1}^p \mathbf{1}\{\hat{\beta}_i \neq 0, \frac{ {q}_0(\hat{\beta}_i)}{{q}(\hat{\beta}_i)} \leq t \}=
\sum_{i=1}^p \mathbf{1}\{\frac{ {q}_0(\hat{\beta}_i)}{{q}(\hat{\beta}_i)} \leq t \}
, 
\end{aligned}
$$
be the number of false rejections and the total number of rejections, respectively, where in both lines of the display the last equality is by Assumption \ref{assmptn:1}. 
Define 
$$
\begin{aligned}
t_{V_m}(x, y) &:= \mathbf{1}\{x \in (a_m(t), b_m(t))\} \cdot \mathbf{1}\{y = 0\}, \quad\quad 
t_{R_m}(x, y) &:= \mathbf{1}\{x \in (a_m(t), b_m(t))\}, 
\end{aligned}
$$
for the functions $a_m, b_m$ from Assumption \ref{assmptn:1}, 
so that 
\[
V^*(t;\lambda) = \sum_{m=1}^M \sum_{i=1}^p t_{V_m}(\hat{\beta}_i, \beta_i), \quad \quad  R^*(t;\lambda) = \sum_{m=1}^M \sum_{i=1}^p t_{R_m}(\hat{\beta}_i, \beta_i).
\]
The following lemma is the crucial step in obtaining asymptotic predictions for the oracle procedure. 
Its proof uses Lemmas \ref{lem:lina-1.2}, \ref{lem:lina-1.3} and \ref{lem:lina-1.4}, which are included in the Appendix. 
The main challenge in the proof arises from the fact that the Lasso estimates are not actually i.i.d.; consequently, the law of large numbers  cannot be applied directly to obtain the desired limits, and instead we must resort to Theorem \ref{sec:AMP-framework} of \cite{bayati2012lasso}. 
In applying that theorem, we sidestep the discontinuity of the indicator functions involved in the expressions for FDP and TPP following a similar approach to  \cite{su2017false}. 

\begin{lemma}
\label{lem:lina-1.1}
Under Assumption \ref{assmptn:1} and the asymptotic setup of Section \ref{sec:AMP-framework}, we have 
\[
\operatorname{FDP}^*(t;\lambda) := \frac{V^*(t;\lambda)}{R^*(t;\lambda) \vee 1}
\ \longrightarrow\ 
\frac{\mathbb{E} \sum_{m=1}^M t_{V_m}(\eta_{\alpha\tau}(\Pi + \tau Z), \Pi)}
     {\mathbb{E}\sum_{m=1}^M t_{R_m}(\eta_{\alpha\tau}(\Pi + \tau Z), \Pi)}, 
\]
in probability as $p\to \infty$, 
where $\tau > 0$, $\alpha > \alpha_{\min}$ are the unique solution to equations \eqref{eq:tau}-\eqref{eq:lambda}.    
\end{lemma}



Theorem \ref{thm:predictions-ol} below gives the limits of FDP and TPP for the oracle procedure \eqref{eq:s_hat_ol}. 
With Lemma \eqref{lem:lina-1.1} in place, it follows from the 
results of \citet{bayati2012lasso} and is analogous to  the results that \cite{bogdan2013supplement} obtained for Lasso selection. 

\begin{theorem}
\label{thm:predictions-ol}
    Under the asymptotic framework of Section \ref{sec:AMP-framework}, fix $\lambda$ and a threshold $t\leq t_0$, where $t_0$ is the constant from  Assumption \ref{assmptn:1}, 
$\operatorname{FDP}^*(t; \lambda)$ and $\operatorname{TPP}^*(t; \lambda)$ both have limits in probability as $p\to \infty$, which we denote 
$$
\begin{aligned}
\operatorname{fdp}^*(t; \lambda) &:= \lim \operatorname{FDP}^*(t; \lambda) = 
\lim \; \frac{
\sum_{i=1}^{p} \mathbf{1}\left\{\frac{q_0(\hat{\beta}_i)}{q(\hat{\beta}_i)} \leq t,\ \beta_i = 0 \right\}
}{
\sum_{i=1}^{p} \mathbf{1}\left\{ \frac{q_0(\hat{\beta}_i)}{q(\hat{\beta}_i)} \leq t \right\}
}, \\[2ex]
\operatorname{tpp}^*(t; \lambda) &:= \lim \operatorname{TPP}^*(t; \lambda) = 
\lim \; \frac{\sum_{i=1}^{p} \mathbf{1}\left\{\ \frac{q_0(\hat{\beta}_i)}{q(\hat{\beta}_i)} \leq t,\ \beta_i \neq 0 \right\}
}{\sum_{i=1}^{p} \mathbf{1}(\beta_i \neq 0)}.
\end{aligned}
$$
Furthermore, these limits are given by 
\begin{align}
\label{eq:fdp-oracle}
\operatorname{fdp}^*(t; \lambda) &= \frac{(1-\varepsilon) \mathbb{P}(\frac{q_0(\eta_{\alpha\tau}(\tau Z))}{q(\eta_{\alpha\tau}(\tau Z))} \leq t )}{(1-\varepsilon) \mathbb{P}( \frac{q_0(\eta_{\alpha\tau}(\tau Z))}{q(\eta_{\alpha\tau}(\tau Z))} \leq t ) + \varepsilon \mathbb{P}(\frac{q_0(\eta_{\alpha\tau}(\Pi_1 +\tau Z))}{q(\eta_{\alpha\tau}(\Pi_1 +\tau Z))} \leq t )}
%
\end{align}
and 
\begin{equation}
\label{eq:tpp-oracle}
\operatorname{tpp}^*(t; \lambda) = \mathbb{P}\bigg(\frac{q_0(\eta_{\alpha\tau}(\Pi_1 +\tau Z))}{q(\eta_{\alpha\tau}(\Pi_1 +\tau Z))} \leq t \bigg). 
%
\end{equation}
\end{theorem}

Before analyzing the EB procedure, we state the following result which formalizes the sense in which the oracle rule \eqref{eq:s_hat_ol} is asymptotically optimal. 
It says that the procedure \eqref{eq:s_hat_ol} is indeed an oracle, not only because it depends on unknown parameters but also in the usual sense that it is (asymptotically) optimal in the class \eqref{eq:T-h-lambda} for any $\lambda$. 

\begin{theorem}
\label{thm:phi-optimal}
Fix $\lambda>0$ and suppose Assumption \ref{assmptn:1} holds with some constant $t_0$. 
For a fixed threshold $t^*\leq t_0$ denote $\xi = \tpp^*(t^*; \lambda)$, the asymptotic TPP of the oracle procedure \eqref{eq:s_hat_ol}. 
Now consider selecting according to \eqref{eq:threshold-selection} for a fixed threshold $t$, where $T_i$ given by \eqref{eq:T-h-lambda} for some function $\varphi$, and let $\fdp^\varphi(t; \lambda)$ and $\tpp^\varphi(t; \lambda)$ be the corresponding limits of  the FDP and TPP of the selection rule, when they exist. 
If $\varphi$ is such that Theorem 1.5 of \citet{bayati2012lasso} applies to $\psi(x,y)=\mathbf{1}(\varphi(x)\geq t)\mathbf{1}(y=0)$ and $\psi(x,y)=\mathbf{1}(\varphi(x)\geq t)$, then for any procedure in the aforementioned class, $\tpp^\varphi(t;\lambda)=\xi$ implies 
$\fdp^*(t^*;\lambda)\leq \fdp^{\varphi}(t;\lambda)$. 
\end{theorem}


\smallskip
We turn to the analysis of the fully data-driven, EB selection procedure. 
To establish the corresponding limits of FDP and TPP under our asymptotic framework it is essential to prove that the estimator $\hat{q}(x)$ converges in probability to $q(x)$ uniformly over \([\Delta, \infty)\) and \((-\infty, -\Delta ]\), $\Delta$  defined in Assumption \ref{assmptn:1}. 
We start by showing pointwise convergence in Proposition \ref{prop:lina-2.1}, with the proof relying crucially on the following lemma. 

\begin{lemma}
\label{lem:lina-2.1.1}
Let 
\[
k(y, x) = \frac{1}{h}\mathcal{K}\left(\frac{x - y}{h}\right) = \frac{1}{\sqrt{2\pi h^2}} e^{-\frac{1}{2}(y - x)^2 / h^2}, \quad h > 0, 
\]
obtained for $\mathcal{K} = \phi$ the standard Gaussian kernel. 
Define
\( \psi(\hat{\beta}_i, \beta_i) = k(\hat{\beta}_i, x) \), with  \(h, x\) fixed. Then 
\( \psi: \mathbb{R} \times \mathbb{R} \to \mathbb{R} \) is a pseudo-Lipschitz function, and 
\[
\frac{1}{p} \sum_{i=1}^{p} \mathbf{1}\{\hat{\beta}_i \neq 0\} \, \psi(\hat{\beta}_i, \beta_i) \longrightarrow \mathbb{E}\left[\mathbf{1}\{\eta_{\alpha \tau}(\Pi + \tau Z) \neq 0\} \, \psi(\eta_{\alpha \tau}(\Pi + \tau Z), \Pi)\right]
\]
in probability as $p\to \infty$. 
\end{lemma}

\begin{proposition}
\label{prop:lina-2.1}
In the asymptotic framework of Section~\ref{sec:AMP-framework}, and under Assumption~\ref{assmptn:1}, 
\[
\hat{q}(x) \longrightarrow q(x)
\]
in probability as $p\to \infty$.
\end{proposition}

To extend this to uniform convergence of $\hat{q}(x)$ we will need an additional assumption. 

\begin{assumption}
\label{assmptn:2}
There is a constant $\mathcal{L} >0$ such that the unnormalized density $q(x)$ is uniformly continuous on the intervals \([-\mathcal{L}, -\Delta]\) and  \([\Delta, \mathcal{L}]\), where $\Delta>0$ is the constant from Assumption \ref{assmptn:1}.
\end{assumption}

\begin{proposition}
\label{prop:lina-2.2.a}
Under Assumptions \ref{assmptn:1} and \ref{assmptn:2}, the kernel density estimator $\hat{q}(x)$ defined in \eqref{eq:w-hat-q-hat} for $\mathcal{K}(x)=\phi(x) = \frac{1}{\sqrt{2\pi}} e^{-x^2 / 2}$, converges uniformly in probability to \( q(x)\) on \([\Delta, \mathcal{L}]\), i.e., 
\[
\sup_{x \in [\Delta, \mathcal{L}]} \left| \hat{q}(x) - q(x) \right| \longrightarrow 0 
\]
in probability as $p\to \infty$. 
\end{proposition}
\begin{remark}
    The uniform convergence in Proposition~\ref{prop:lina-2.2.a} holds also on the interval \([-\mathcal{L}, -\Delta ]\).
\end{remark}

We are now in a position to derive the asymptotic predictions for the EB procedure. 

\begin{theorem}
\label{thm:predictions-eb}
    In the asymptotic framework of Section \ref{sec:AMP-framework} and under Assumptions \ref{assmptn:1} and \ref{assmptn:2}, fix $\lambda$ and a threshold $t\leq t_0$, where $t_0$ is the constant from  Assumption \ref{assmptn:1}. 
    Then 
$\operatorname{FDP}^{EB}(t; \lambda)$ and $\operatorname{TPP}^{EB}(t; \lambda)$ both have limits in probability as $p\to \infty$, which we denote 
$$
\begin{aligned}
\operatorname{fdp}^{EB}(t; \lambda) &:= \lim \operatorname{FDP}^{EB}(t; \lambda) = 
\lim \; \frac{
\sum_{i=1}^{p} \mathbf{1}\left\{\hat{\beta}_i \neq 0,\ \frac{\hat{q}_0(\hat{\beta}_i)}{\hat{q}(\hat{\beta}_i)} \leq t,\ \beta_i = 0 \right\}
}{
\sum_{i=1}^{p} \mathbf{1}\left\{\hat{\beta}_i \neq 0,\ \frac{\hat{q}_0(\hat{\beta}_i)}{\hat{q}(\hat{\beta}_i)} \leq t \right\}
}, \\[2ex]
\operatorname{tpp}^{EB}(t; \lambda) &:= \lim \operatorname{TPP}^{EB}(t; \lambda) = 
\lim \; \frac{\sum_{i=1}^{p} \mathbf{1}\left\{ \hat{\beta}_i \neq 0,\ \frac{\hat{q}_0(\hat{\beta}_i)}{\hat{q}(\hat{\beta}_i)} \leq t,\ \beta_i \neq 0 \right\}
}{\sum_{i=1}^{p} \mathbf{1}(\beta_i \neq 0)}.
\end{aligned}
$$
Furthermore, it holds that
\begin{equation}
\label{fdp-eb}
\operatorname{fdp}^{EB}(t; \lambda) = \operatorname{fdp}^*(t; \lambda),\quad\quad\quad \operatorname{tpp}^{EB}(t; \lambda) = \operatorname{tpp}^*(t; \lambda), 
\end{equation}
where $\operatorname{fdp}^*(t; \lambda)$ and $\operatorname{fdp}^*(t; \lambda)$ are the formulas in \eqref{eq:fdp-oracle} and \eqref{eq:tpp-oracle}, respectively. 
\end{theorem}

\begin{remark}
Theorem \ref{thm:predictions-eb} holds more generally for any kernel $\mathcal{K}$ which is Lipschitz continuous, i.e., if there exists a constant \( L \) such that $|\mathcal{K}(x) - \mathcal{K}(y)| \leq L |x - y|$ for all $x, y \in \mathbb{R}$. 
This condition is satisfied for the Gaussian kernel used above. 
\end{remark}

Theorem \ref{thm:predictions-eb} says that for any fixed $t$, the oracle rule \eqref{eq:s_hat_ol} and the EB competitor \eqref{eq:s_hat_eb} have the same FDP and TPP asymptotically; in other words, under our assumptions, the theory predicts that the tradeoff curves for the two procedures coincide. 
This result is encouraging because it says that for any fixed $\lambda$ the performance of the optimal procedure in the class \eqref{eq:T-h-lambda} is asymptotically attainable by a legitimate procedure. 
Still, the performance depends on the choice of the regularization parameter $\lambda$, and different values of $\lambda$ can give very different tradeoff curves (as is the case for thresholded-Lasso). 
For each distribution $\Pi$, 
it is then natural to ask what value of $\lambda$ is optimal, in the sense of minimizing the asymptotic FDP at a given asymptotic TPP level. 
For thresholded-Lasso selection, 
\cite{wang2020bridge} proved that the optimal value of $\lambda$ is the minimizer of the asymptotic mean squared error, 
\begin{equation}
\label{eq:lambda-opt}
\lambda^* := \argmin_{\lambda} \EE [ (\eta_{\alpha\tau}(\Pi + \tau Z)-\Pi)^2 ], 
\end{equation}
and from \eqref{eq:tau}, it is immediate that $\lambda^*$ is equivalently given by the  minimizer of $\tau$ in $\lambda$. 
Our next result asserts that $\lambda^*$ defined above is also optimal for the oracle \eqref{eq:s_hat_ol} that selects based on the lfdr. 
In the statement of the theorem below and in its proof, whenever an asterisk ($*$) superscript appears, it refers to the quantity when $\lambda=\lambda^*$, and if there is no superscript it refers to an arbitrary $\lambda\neq \lambda^*$; for example, $\lfdr^*(x)$ is the lfdr in \eqref{lfdr} computed under $\lambda=\lambda^*$, and $\lfdr(x)$ is this quantity for an arbitrary $\lambda\neq \lambda^*$. 

\begin{theorem}
\label{thm:optimal-lambda}
Under Assumption~\ref{assmptn:1}, let $t_0^*$ be the value of $t_0$ when  $\lambda^*$. 
For some $t^*\leq t_0^*$ consider the oracle that selects when $\lfdr^*(\hat{\beta}_i)\leq t^*$, and let $\xi = \tpp^*(t^*; \lambda^*)$ be the corresponding asymptotic TPP. 
Suppose there exists another $\lambda$, and a corresponding fixed $t$, such that the oracle which selects when $\lfdr(\hat{\beta}_i)\leq t$ has $\tpp^*(t;\lambda)=\xi$. 
For any such $\lambda$, we have 
$$
\fdp^*(t^*; \lambda^*)\leq \fdp^*(t; \lambda). 
$$
\end{theorem}

It is important to note that Theorem \ref{thm:optimal-lambda} implies a stronger oracle than that of Theorem \ref{thm:phi-optimal}. 
Specifically, 
if $\mathcal{P}$ denotes the collection of univariate functions $\varphi$ satisfying the condition in Theorem \ref{thm:phi-optimal}, then 
by combining Theorem \ref{thm:optimal-lambda} with Theorem \ref{thm:phi-optimal} and using the fact that the oracle and EB procedures have the same asymptotic tradeoff curves by Theorems \ref{thm:predictions-ol} and \ref{thm:predictions-eb}, we conclude that for any pair $(\varphi, \lambda)$ such that $\varphi \in\mathcal{P}$, if $\tpp^*(t^*; \lambda^*) = \tpp^\varphi(t;\lambda)$ then $\fdp^*(t^*; \lambda^*)\leq \fdp^{\varphi}(t; \lambda)$. 
That is, if we fix the asymptotic TPP level, then the oracle procedure with $\lambda^*$ achieves the minimum asymptotic FDP among all choices of both $\varphi$ and $\lambda$  in \eqref{eq:T-h-lambda}, while Theorem \ref{thm:phi-optimal} only says that for any {\em fixed} value of $\lambda$ it achieves the minimum among  $\varphi\in \mathcal{P}$. 
This has important consequences because the extended class, allowing to  search also over $\lambda$, 
includes the Lasso selection procedure \eqref{eq:lasso-selection}, which implies that at $\lambda^*$ the asymptotic FDP of the EB procedure is necessarily smaller than Lasso if we fix the TPP level (this does not rule out the possibility that for some other $\lambda$ Lasso is better than EB). 
In the last section of the paper we explain why we chose to state the optimality result in Theorem \ref{thm:phi-optimal} for an arbitrary fixed $\lambda$ instead of the optimal value $\lambda^*$.

\smallskip
In the statement of Theorem \ref{thm:optimal-lambda} above, the assertion that the result holds also for the EB procedure is trivially implied by the fact that the tradeoff curves of these two procedures coincide, under our assumptions. 
Still, the optimal value $\lambda^*$ depends on unknown quantities of the problem, i.e., it is again an oracle value and cannot be used in practice. 
The characterization of $\lambda^*$ in \eqref{eq:lambda-opt}, however, suggests that it can be well estimated by cross-validation. 
Specifically, let $\hat{\lambda}_{\cv}$ denote the estimator of $\lambda$ obtained by $K$-fold cross-validation. 
In Lemma 4.1 of \citet{weinstein2023power} the limit of $\hat{\lambda}_{\cv}$ was derived, under the AMP asymptotic framework, in a setting where $\bX$ is augmented with a knockoffs matrix. 
Adapting this result to the current setting (no knockoffs involved), we conclude immediately the following. 

\begin{proposition}[consequence of Lemma 4.1 in \citet{weinstein2023power}]
\label{prop:lambda-cv}
Under our asymptotic framework, we have $\hat{\lambda}_{\cv} \longrightarrow  \lambda^*_{\cv}$ in probability as $p\to \infty$, where 
\begin{equation}
\label{eq:lambda-cv}
\lambda^*_{\cv}\ := \argmin_\lambda \tau(\lambda ;(K-1) \delta / K). 
\end{equation}
\end{proposition}
Additionally, by adapting Equation (4.2) in \citet{weinstein2023power}, we conclude that 
$$
\lambda^*_{cv} = \alpha_\cv \tau_\cv \bigg[ 1-\frac{K}{(K-1)\delta} \PP(|\Pi + \tau_\cv Z|>\alpha_\cv \tau_\cv)\bigg],
$$
where $\alpha_\cv$ and $\tau_\cv$ solve the system given by 
\begin{equation}
\label{eq:cv-amp}
\begin{aligned}
&\tau^2_\cv = \sigma^2 + \frac{K}{(K-1)\delta} \EE[\eta_{\alpha_\cv \tau_\cv}(\Pi + \tau_\cv Z) - \Pi]^2\\
&\EE(Z+\alpha_\cv ; \Pi + \tau_\cv Z < -\alpha_\cv \tau_\cv) - \EE(Z - \alpha_\cv ; \Pi + \tau_\cv Z > \alpha_\cv \tau_\cv)=0. 
\end{aligned}
\end{equation}

\section{Numerical experiments}
\label{sec:numerical}
We now demonstrate the adequacy of the theoretical predictions from the previous section through some numerical studies. 
First we carry out simulations to 
examine the agreement of the empirical tradeoff curves of the oracle and EB procedures with the theoretical predictions. 
In Figure \ref{fig:simulations}, each of the subplots displays empirical tradeoff curves, that is, the true FDP and TPP along the realized selection path, from 17 independent simulation runs under different settings, where we fixed $\lambda=1$. In all subplots the data $(Y,\bX)$ was generated from the two-level model given by \eqref{eq:linear_model} and \eqref{eq:iid-two-groups} with $\sigma = 1$ and $\epsilon=0.1$. 
The left column shows results for {$p=5000, n=2500$} (corresponding to $\delta=0.5$), and the right column shows results for {$p=5000, n=9000$} (corresponding to $\delta=1.8$). 
Each row in the figure corresponds to a different choice of the nonzero component $\Pi_1$: in the first row $\Pi_1$ is Gaussian with mean 3.5 and variance 1; in the second row $\Pi_1$ is a mixture of two Gaussians, one with mean -3.6 and variance 1 and the second with mean 4 and variance 1, 
\begin{figure}[]
    \centering \includegraphics[width=\textwidth]{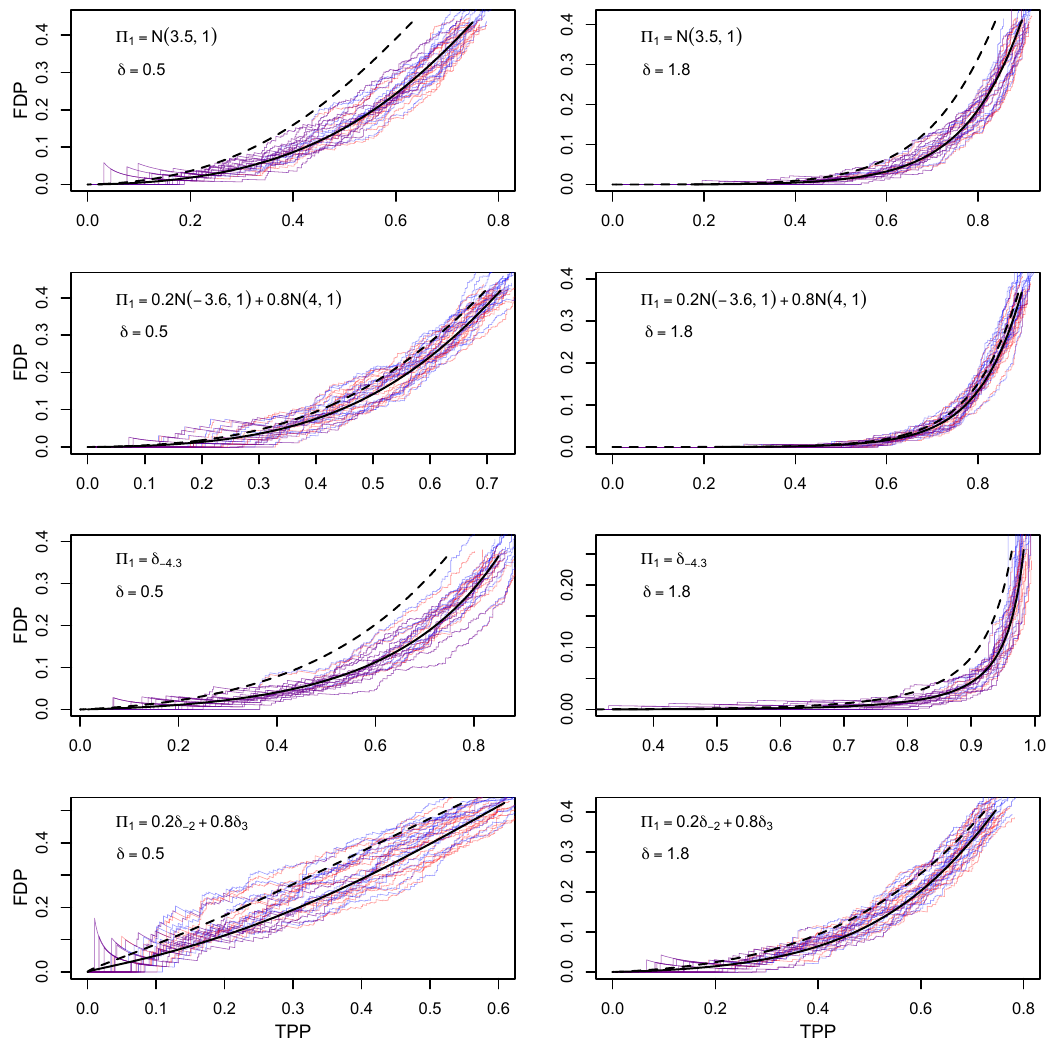}
    \caption{Empirical tradeoff curves vs.~theoretical predictions  for the oracle and the EB procedure. Left column displays results for $\delta=0.5$, right column for $\delta=1.8$. 
    The rows correspond to different choices of $\Pi_1$. 
    Thin lines are realized FDP and TPP from 17 simulation runs, red for the oracle and blue for the EB procedure. Solid black lines represent theoretical predictions $\fdp^*(t), \tpp^*(t)$. Broken black line is the theoretical curve for thresholded-Lasso. 
    Further details are included in the main text. 
    }
    \label{fig:simulations}
\end{figure}
with probabilities 0.2 and 0.8 respectively;  
in the third row $\Pi_1$ is a point mass at -4.3; and in the fourth row it is a mixture of two point mass distributions, one at -2 and the other at 3, with probabilities 0.2 and 0.8, respectively. 
The thin blue curves are the empirical tradeoff curves for the EB procedure, and the thin red curves are for the oracle procedure, which uses the knowledge of $\Pi_1, \epsilon$ and $\sigma$ to rank the variables according to the true (asymptotic) lfdr. 
In each of the subplots the theoretical tradeoff curve, predicting the asymptotic FDP and TPP for the oracle and the EB procedure, is overlaid in black. 
For comparison, the broken black line is the theoretical tradeoff curve for thresholded-Lasso. 
The plots show good agreement between the theory and the empirical results. 
In particular, the empirical curves of the oracle and the EB procedure are difficult to distinguish, which is aligned with the theory that predicts they coincide asymptotically. 

In Figure \ref{fig:lambda-opt} we examine the effect of $\lambda$ on power of the EB selection rule. 
The left panel shows the asymptotic FDP of the EB procedure as a function of $\lambda$, when the asymptotic TPP is fixed at 0.7 (an arbitrary choice), for an example with $\epsilon = 0.1, \delta = 1, \sigma =1$ and $\Pi_1 = .2\calN(-3.6, 1) + .8\calN(4,1)$; 
that is, the graph displays $\fdp^*(t(\lambda); \lambda)$ vs.~$\lambda$, where $t(\lambda)$ is such that $\tpp^*(t(\lambda); \lambda)=0.7$. 
The dotted vertical line indicates the optimal value $\lambda^*$, specified in \eqref{eq:lambda-opt}, and the solid vertical line indicates $\lambda^*_{cv}$ in \eqref{eq:lambda-cv}. 
The picture is consistent with the theory of the previous section, which predicts that $\lambda^*$ minimizes the asymptotic FDP for any fixed TPP level. 
The right panel of the figure presents results from a corresponding simulation, where we examine the convergence predicted in Proposition \ref{prop:lambda-cv} of $\hat{\lambda}_\cv$ to $\lambda^*_{cv}$. 
For $p\in \{200, 2000, 10^4\}$ we set $n=p$ and report the mean and standard error of $\hat{\lambda}_\cv$ based on 1000 independent runs for $p=200, 2000$, and based on 100 runs for $p=10^4$. 
To inspect the asymptotic loss in power expected due to using the EB procedure with $\hat{\lambda}_\cv$ instead of the optimal choice of $\lambda$, 
Figure \ref{fig:fdp-cv_fdp_opt} compares the asymptotic FDP for $\lambda^*_{cv}$ to the asymptotic FDP for $\lambda^*$, when both are calibrated to the same asymptotic TPP level. 
That is, if for any $\lambda$ we let $t_\xi(\lambda)$ be such that $\tpp^*(t_\xi(\lambda); \lambda)=\xi$, then the graphs show $\fdp^*(t_\xi(\lambda^*_{cv}); \lambda^*_{cv})$ versus $\fdp^*(t_\xi(\lambda^*); \lambda^*)$ as the TPP level $\xi$ varies. 
In the figure $\delta=1$, and the different panels correspond to the four different choices of $\Pi_1$ in Figure \ref{fig:simulations}. 
In all panels the points closely follow the identity line, demonstrating that the power loss is negligible asymptotically.

\begin{figure}
\centering

\begin{subfigure}[t]{0.48\textwidth}
  \centering
  \includegraphics[width=.9\linewidth]{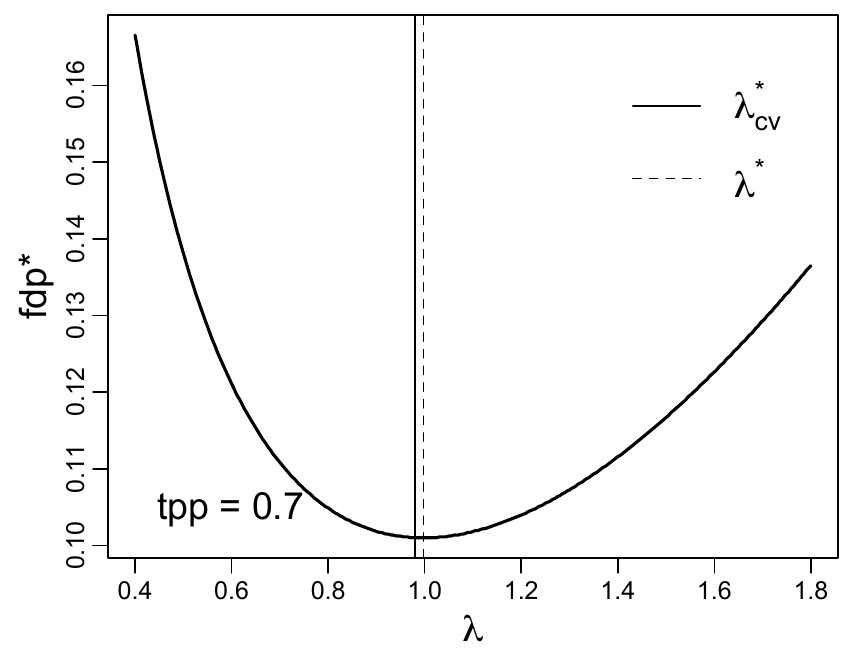}
  \label{fig:panel-a}
\end{subfigure}\hfill
\begin{subfigure}[t]{0.48\textwidth}
  \centering
  \includegraphics[width=.9\linewidth]{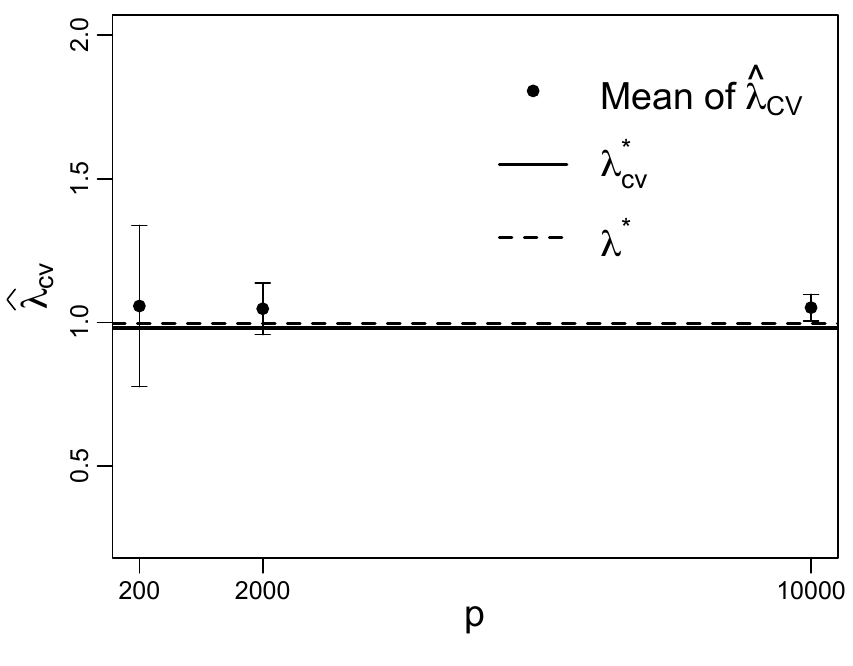}
  \label{fig:panel-b}
\end{subfigure}
\vspace{5pt}
\caption{Left: asymptotic FDP vs.~$\lambda$ when fixing the asymptotic TPP at $0.7$, for a setting with $\epsilon = 0.1, \delta = 1, \sigma =1$ and $\Pi_1 = .2\calN(-3.6, 1) + .8\calN(4,1)$. Right: simulation mean and standard error of $\hat{\lambda}_\cv$ for different $p$.} 
\label{fig:lambda-opt}
\end{figure}

\begin{figure}
    \centering \includegraphics[width=.8\textwidth]{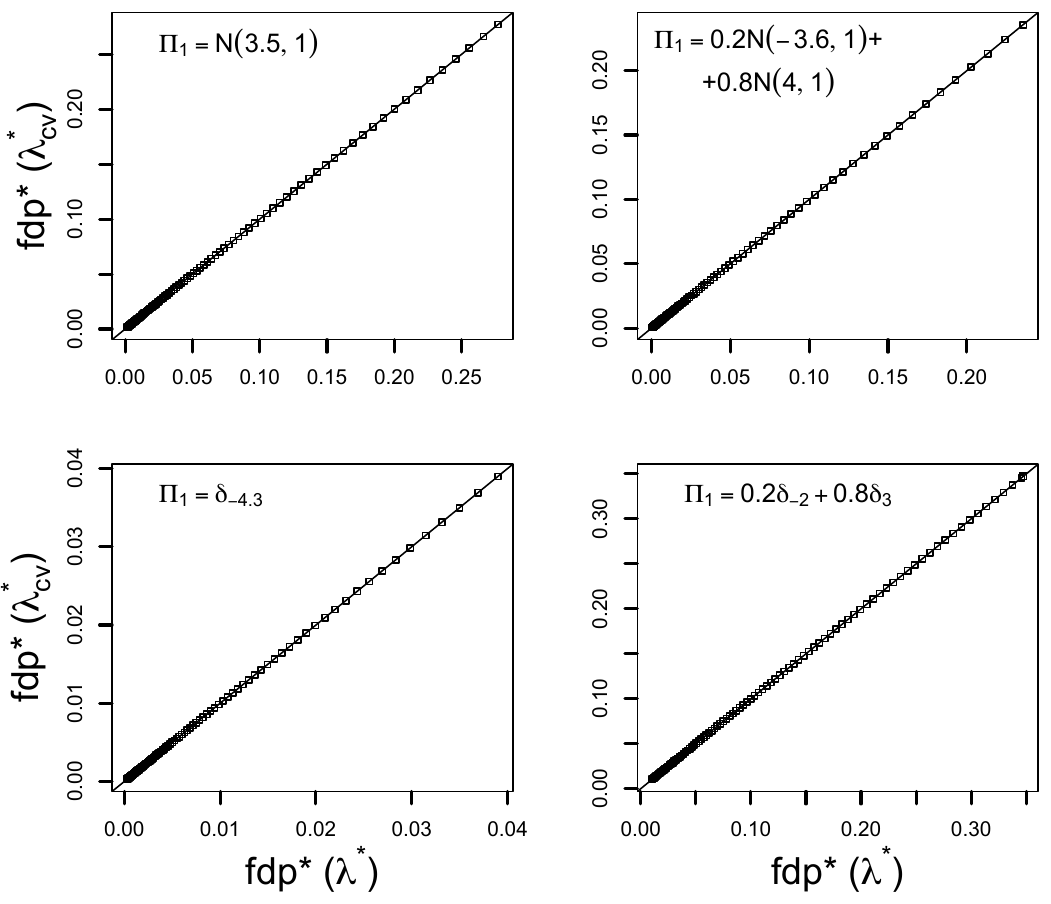}
    \vspace{6pt}
    \caption{Asymptotic FDP of the EB procedure when using $\lambda^*_{cv}$ vs.~the optimal value $\lambda^*$ for $\delta=1$ and four different choices of $\Pi_1$. In each panel, every point corresponds to some fixed value $\xi$ of the (asymptotic) TPP level. The identity line is shown for reference. 
    }
    \label{fig:fdp-cv_fdp_opt}
\end{figure}


\section{Estimating the $\lfdr$: beyond power considerations}
\label{sec:local-fdr}

Up to this point, our interest in estimating the local false discovery rate \eqref{lfdr} has been motivated purely by power considerations, specifically,  by the fact that it emerges as the (oracle) choice of the importance statistic that minimizes asymptotic FDP 
for any fixed asymptotic TPP level over all procedures in the class \eqref{eq:T-h-lambda}. 
We now move beyond power considerations and explain why estimating the lfdr is of interest in its own right. 
The viewpoint adopted here is inspired by the ideas introduced in \cite{soloff2024edge, xiang2025frequentist}. 
At a high level, \cite{soloff2024edge} raise a conceptual objection to using the FDR as a type-I error criterion in multiple testing (under a two-groups model), because the (Bayesian) FDR of a procedure only accounts for the rate of errors incurred on the set of rejections collectively, while a guarantee on the error rate of each individual rejection is arguably more appropriate. 
To be more specific and apply this in our setting, the FDR of a procedure that selects the $i$th variable if $T_i\leq t$ for some statistic $T_i = \varphi(\hat{\beta}_i)$ and a fixed $t$, can be shown, using the towering property of expectation,\footnote{If the pairs $(\hat{\beta}_i, \beta_i)$ were actually i.i.d.~from \eqref{eq:two-groups}, we would have $\operatorname{pFDR}=\FDR/\PP(R>0) = \EE[\lfdr(\beta_i)\lvert T_i\leq t]$ for any finite $p$ \citep{efron2008microarrays}, where $R$ is the total number of rejections; this is not true in our setting, but the asymptotic claim made in the text still holds.} to  asymptotically equal the conditional {\em expectation} of $\lfdr(\beta_i)$ given the rejection event $T_i\leq t$. 
This implies that some rejections---usually those with $T_i$ near the rejection boundary $t$---could have lfdr values substantially exceeding the nominal FDR level (while others could have much smaller lfdr). 
To fix this deficiency, \cite{soloff2024edge} proposed to use instead an error criterion that directly targets the lfdr, for example, to aim for controlling the expectation of the maximum lfdr over the rejections (``max-lfdr"). 
Furthermore, they introduced the {\em Support Line} (SL) procedure, which operates on p-values and controls the max-lfdr in finite sample, assuming i.i.d.~two-groups model and a monotonicity condition on the density of the alternative (a somewhat weaker guarantee holds, still in finite sample, without the monotonicity condition). 
These finite sample results do not formally apply to the more complicated problem of controlled variable selection under the regression model \eqref{eq:linear_model}: even if individual p-values that marginally follow a two-groups model can be calculated, they are generally not independent. 
Finding a procedure in the variable selection problem  that provably has finite sample control of max-lfdr or similar lfdr-based error metrics in fact  seems highly nontrivial. 

That being said, in the scope of the current paper, 
suppose that we have calculated the Lasso estimates $\hat{\beta}_1,...,\hat{\beta}_p$ for some fixed $\lambda$. 
Instead of considering the FDR of a variable selection procedure that ranks the variables according to some function of $\hat{\beta}_i$, we may consider the statistics $\hat{\beta}_i$ themselves, and, without yet making any decision, associate each with the corresponding lfdr \eqref{sec:local-fdr}. 
The lfdr is an unknown quantity, because its calculation requires knowledge of $\Pi$ (and $\sigma$), and in fact it would have the formal meaning as the posterior probability $P(\beta_i = 0 | \hat{\beta}_i)$ only if $(\hat{\beta}_i, \beta_i)$ were actually distributed (marginally) according to \eqref{eq:two-groups}, which, as we remarked before, is not the case here. 
Still, under the special AMP setup, similar calculation as in Section \ref{sec:theoretical} yield asymptotic predictions for the fraction of true nulls among variables with $\hat{\beta}_i \in [s,t]$ for fixed $0<s<t$, 
$$
\FDP\big([s,t]\big) 
:= \frac{\# \{i: s \leq \hat{\beta}_i\leq t, \, \beta_i=0\}}{\# \{i: s \leq \hat{\beta}_i\leq t\}}
\ \longrightarrow \ 
\PP\big(\Pi=0 \big\lvert s \leq \eta_{\alpha\tau}(\Pi + \tau Z)\leq t)\big)
$$
in probability as $p\to \infty$. 
In turn, as the interval length shrinks to zero, the expression on the right hand side converges to the lfdr at $t$, i.e., 
$$
\PP\big(\Pi=0 \big\lvert s \leq \eta_{\alpha\tau}(\Pi + \tau Z)\leq t)\big)
 \longrightarrow 
\PP\big(\Pi=0 \big\lvert \eta_{\alpha\tau}(\Pi + \tau Z) = t)\big)  = \frac{(1 - \epsilon) f_0(\hat{\beta}_i)}{f(\hat{\beta}_i)} = \lfdr(t)
$$
as $s\to t^-$. 
Therefore, we have 
\begin{equation}
\label{eq:FDP-interval}
\lim_{s\to t^-} \lim_{p\to \infty}\FDP\big([s,t]\big) = \lfdr(t), 
\end{equation}
that is, $\lfdr(t)$ is the limiting fraction of true nulls among variables with $\hat{\beta}_i$ in a shrinking  interval around $t$. 
Note that the function $\lfdr(\cdot)$ in \eqref{lfdr} is not claimed to represent the actual posterior probability that $\beta_i=0$ given $\hat{\beta}_i$ (again, this interpretation is only valid if $(\hat{\beta}_i, \beta_i)$ are identically distributed according to \eqref{eq:two-groups}). 
Still, \eqref{eq:FDP-interval} lends the lfdr in \eqref{lfdr} an interpretation under the AMP framework. 
Consistent with our comment at the beginning of Section \ref{sec:AMP-framework}, we note also that \eqref{eq:FDP-interval} is meaningful even when $\beta_i$ are fixed, if the conditions on their empirical distribution are satisfied. 

If we had access to the true function $\lfdr(t)$, this would allow moving  beyond FDR assessment for variable selection procedures, toward a more exploratory form of analysis. 
Indeed, imagine we observed the unlabeled vector $\hat{\beta}$, i.e.~the values $\hat{\beta}_i$ but not their ordering. 
If we highlight a subset of these values---completely {\em post-hoc}---as locations `of potential interest', then \eqref{eq:FDP-interval} suggests that, at any of these locations, the lfdr is roughly the chance that it corresponds to a true null $\beta_i=0$. 
The lfdr is unknown, but we can estimate it, e.g., as
$$
\widehat{\lfdr}(t) = \frac{\widehat{1-\epsilon}\hat{f}_0(t)}{\hat{f}(t)}, \quad\quad \widehat{1-\epsilon} = \frac{N(\lambda)}{p[1-2 \Phi(-\hat{\alpha})]}, 
$$
where $N(\lambda)=\{ i \leq p: \hat{\beta}_i(\lambda)=0 \}$;  $\hat{f}$ and $\hat{f}_0$ are given by \eqref{f_hat} and \eqref{f_0_hat}, respectively; and $\hat{\alpha}$ is obtained from \eqref{tau-alphatau-hat}. 
Under our asymptotic setup and the assumptions in Section \ref{sec:theoretical}, we have 
$$
\begin{aligned}
\lim_{p\to\infty} \frac{N(\lambda)}{p} = \PP\left(|\Pi+\tau Z|\leq \alpha \tau \right) &= (1-\epsilon) \PP(|Z| \leq \alpha)+\epsilon \PP\left(\left|\Pi_1 + \tau Z\right| \leq \alpha \tau\right) \\
&=(1-\epsilon)[1-2 \Phi(-\alpha)] + \epsilon \PP\left(\left|\Pi_1 + \tau Z\right| \leq \alpha \tau\right)
\end{aligned}
$$
and 
\begin{equation}
\label{eq:lfdr-hat-predicted}
\begin{aligned}
\lim_{p\to\infty} \widehat{\lfdr}(t) \ = \ 
\lfdr(t) \ + \ \frac{\epsilon \PP\left(\left|\Pi_1 + \tau Z\right| \leq \alpha \tau\right)}{1-2 \Phi(-\alpha)}
\frac{f_0(t)}{f(t)}, 
\end{aligned}
\end{equation}
which is expected to be only a slight overestimate of $\lfdr(t)$ since $\PP\left(\left|\Pi_1 + \tau Z\right| \leq \alpha \tau\right)$ is typically small. 
For illustration, Figure \ref{fig:lfdr} corresponds to an example with 
$\epsilon = 0.1, \delta = 1, \sigma =1$, $\Pi_1 = .2\calN(-3.6, 1) + .8\calN(4,1)$ and $\lambda=1$. 
The solid curves represent $\lfdr(t)$ (black) and the limit \eqref{eq:lfdr-hat-predicted} of the proposed estimator $\widehat{\lfdr}(t)$ (red); these indeed lie very closely to each other in this example. 
The grey lines in the plot are results from a simulation with $n=p= 14000$, where for each of 17 independent runs we calculated $\FDP\big([t-\Delta,t+\Delta])$, where we use $\Delta = 0.4$,
the realized FDP in a small neighborhood of $t$, for a grid of values of $t$. 
The theoretical predictions align closely with the empirical results; in particular, they capture the differing behavior in the left and right shoulders. 

\begin{figure}[]
\centering
\includegraphics[width=0.65\linewidth]{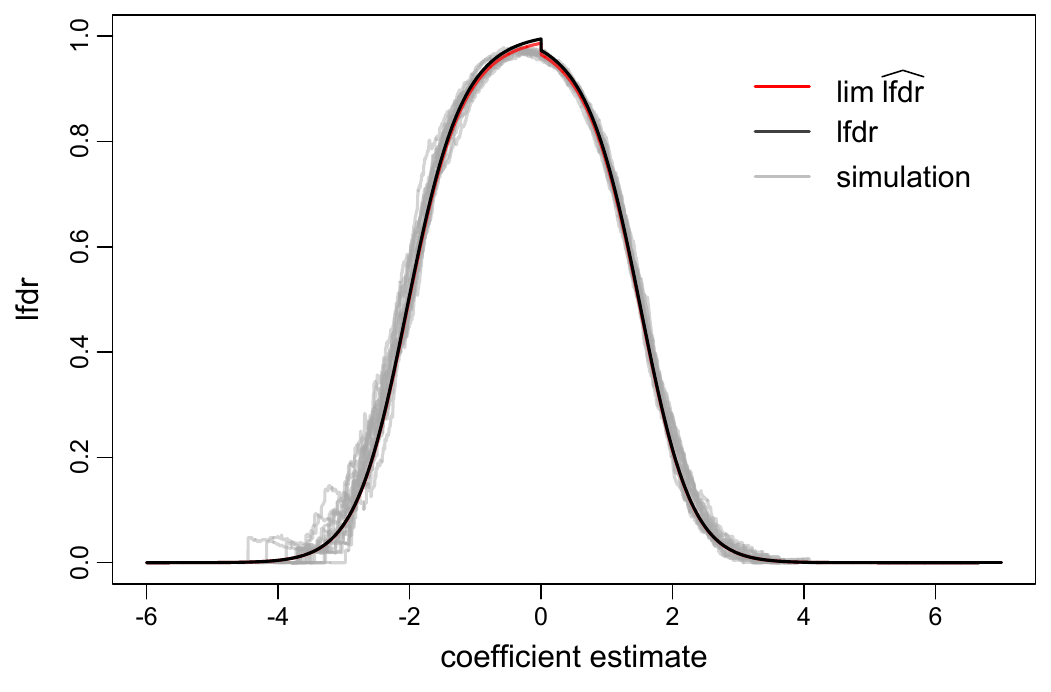}
\vspace{4pt}
\caption{Estimation of the local FDR in an example with $\epsilon = 0.1, \delta = 1, \sigma =1$, $\Pi_1 = .2\calN(-3.6, 1) + .8\calN(4,1)$, and $\lambda=1$ used in calculating the Lasso estimates. 
Solid curves are theoretical asymptotic predictions of the true lfdr and its estimate. 
Grey lines represent realized FDP in a small moving window from 17 independent simulation runs with n=p= 14000, as described in the main text. 
}
\label{fig:lfdr}
\end{figure}

\section{Discussion}
\label{sec:discussion}

We conclude with comments on a few aspects not addressed in the current work. 
Concerning our definition of the oracle, note that 
Theorem~\ref{thm:phi-optimal} asserts optimality over $\varphi$ for a fixed $\lambda$, even though we later identify also the optimal value of $\lambda$ (Theorem \ref{thm:optimal-lambda}). 
The reason we present it this way is that for a fixed $\lambda$ we are able to show 
that the performance of the oracle is asymptotically attainable by a legitimate procedure, namely our EB selection rule. 
A more ambitious result would say that the EB procedure equipped with $\hat{\lambda}_\cv$ (instead of a fixed $\lambda$) asymptotically attains the performance of the oracle that uses $\lambda^*_{cv}$, which by \eqref{eq:lambda-opt} and the numerical examples of Section \ref{sec:numerical} is expected to perform almost as good as using $\lambda^*$. 
While convergence of the EB procedure at $\hat{\lambda}_\cv$ does appear to hold in our simulations, to prove it formally calls for establishing uniform convergence of $\hat{q}_0/\hat{q}$ in $\lambda$; 
in our approach, this requires extending the results for the KDE $\hat{q}$ to uniform convergence in $\lambda$, as well as proving uniform convergence of the estimators $\widehat{\alpha\tau}$ and $\hat{\tau}$ in $\lambda$. 
Such an analysis requires substantial additional technical work, and is therefore not included in the present paper.  

Second, returning to a point discussed in the opening section, we remark  that in practice, the statistician will usually be interested in supplementing the choice of the importance statistic with a choice of a data-dependent threshold $\hat{t}$ that calibrates (or approximately calibrates) selection for type I errors. 
For example, \cite{weinstein2017power} used a simple implementation of knockoffs that in our special setting ($\bX$ has i.i.d.~entries) controls the FDR for any $n,p$; a similar knockoffs analog can be considered for the proposed EB procedure. 
In that case the comparisons to the oracle in the current paper need to be adjusted not only for the fact that the oracle operates on the original $\bX$ whereas the knockoffs will operate on an augmented matrix, but also for potential conservativeness of the knockoffs procedure in choosing the data-dependent threshold. 
More specifically, 
assuming variables with large $T_i$ are selected, in the strategy of \cite{weinstein2017power} the data driven threshold $\hat{t}_q$ is the most liberal $t$ such that $\widehat{\FDP}(t)\leq q$, for a (knockoffs) estimate $\widehat{\FDP}(t)$ of $\FDP(t)$. 
Since $\widehat{\FDP}(t)$ is upwardly biased, $\hat{t}_q$ will generally converge to a value $t_q$ for which $\FDP(t_q)$, the actual FDP, is slightly smaller than the nominal level $q$. This itself generally results in some loss of power, which needs to be accounted for in an honest comparison to the oracle.  
%
Still, the difference is expected to be small asymptotically, especially if $\widehat{\FDP}(t)$ incorporates a reasonably good estimate of $\epsilon$.

\section*{Acknowledgements} 

This work was supported by the Israel Science Foundation under grant no.~2679/24.

\bibliographystyle{abbrvnat}
\bibliography{references}

\bigskip

\appendix

\section{Proofs}
\label{Appendix}

\newcommand{\manuallemma}[2]{%
  \par\noindent\textbf{Lemma #1.}%
   \addcontentsline{toc}{subsection}{Lemma #1}%
  \hypertarget{#2}{}%
}

\newcommand{\manualref}[1]{Lemma~\hyperlink{#1}{#1}}
\newcommand{\manualprop}[2]{%
  \par\noindent\textbf{Proposition #1.}%
  \addcontentsline{toc}{subsection}{Proposition #1}%
  \hypertarget{#2}{}%
}

\newcommand{\manualpropref}[1]{Proposition~\hyperlink{#1}{#1}}


\smallskip
\noindent Before presenting the proof of Lemma~\ref{lem:lina-1.1}, we give another auxiliary result. 
The following is a slightly modified version of Lemma 2 from  \cite{bogdan2013supplement}, suitable for the current setting.

\begin{lemma}
\label{lem:lina-1.2}
Suppose \( Y \) is a \( p \)-dimensional vector distributed as \( \mathcal{N}(\mu, \Sigma) \), where \( \Sigma \geq \sigma^2 I \).  
For any \( \epsilon \in (0,1) \), there exists a constant \( c = c(\epsilon) > 0 \) such that for any \( h > 0 \),
\[
\mathbb{P}(\text{at least } \epsilon p \text{ components of } Y \text{ are in } (b, b+h)) \leq \left( \min \left( 1, c h^\epsilon \sigma^{-\epsilon} \right) \right)^p.
\]
\end{lemma}

\begin{proof}[Proof of Lemma \ref{lem:lina-1.2}]
    By Lemma 1 in \citet{bogdan2013supplement}, the variance of $Y_i \mid Y_{(-i)}$, where the subscript $(-i)$ denotes all components except the $i$th, is larger than or equal to 
$\sigma^2$. So we have
\[
\mathbb{P}(Y_i \in (b, b+h) \mid Y_{(-i)}) \leq \min\left(1, \frac{h}{\sqrt{2\pi}\sigma} \right) \quad \text{a.s.}
\]
since the normal density function is bounded by $1/\sqrt{2\pi}$.
Denote by $\xi_i$ i.i.d. Bernoulli random variables independent of $Y$ with 
\[
\mathbb{P}(\xi_i = 1) = \tilde{p} = \min\left(1, \frac{h}{\sqrt{2\pi}\sigma} \right).
\]
Define $\zeta_i = \mathbf{1}(Y_i \in (b, b+h))$.  
Since 
\[
\mathbb{P}(\zeta_1 = 1 \mid \zeta_2, \ldots, \zeta_p) \overset{\text{a.s}}{\leq} \mathbb{P}(\xi_1 = 1 \mid \zeta_2, \ldots, \zeta_p),
\]
we obtain:
\[
\mathbb{P}(\zeta_1 + \zeta_2 + \cdots + \zeta_p \geq \varepsilon p) \leq \mathbb{P}(\xi_1 + \zeta_2 + \cdots + \zeta_p \geq \varepsilon p).
\]
Using a similar argument, we can conclude that:
\[
\ \ \ \ \ \ \ \  \mathbb{P}(\xi_1 + \xi_2 + \cdots + \xi_k + \zeta_{k+1} + \cdots + \zeta_p \geq \varepsilon p)
\leq 
\mathbb{P}(\xi_1 + \xi_2 + \cdots + \xi_{k+1} + \zeta_{k+2} + \cdots + \zeta_p \geq \varepsilon p)
\]
for $k = 1, \ldots, p-1$. 
Therefore, 
\begin{align*}
\mathbb{P}(\text{at least } \varepsilon p \text{ components of } Y \text{ are in } (b, b+h)) 
&= \mathbb{P}(\zeta_1 + \cdots + \zeta_p \geq \varepsilon p) \\
&\leq \mathbb{P}(\xi_1 + \zeta_2 + \cdots + \zeta_p \geq \varepsilon p) \\
&\leq \mathbb{P}(\xi_1 + \xi_2 + \zeta_3 + \cdots + \zeta_p \geq \varepsilon p) \\
&\quad \vdots \\
&\leq \mathbb{P}(\xi_1 + \xi_2 + \cdots + \xi_p \geq \varepsilon p) \\
&= \mathbb{P}\left( \sum_{i=1}^p \xi_i \geq \varepsilon p \right).
\end{align*}
Hence,
$$
\mathbb{P}\left( \sum_{i=1}^p \xi_i \geq \varepsilon p \right) 
\leq \sum_{i \geq \varepsilon p} \binom{p}{i} \tilde{p}^i 
\leq \tilde{p}^{\varepsilon p} \sum_{i \geq \varepsilon p} \binom{p}{i} 
\leq \tilde{p}^{\varepsilon p} \cdot 2^p\leq \left( \frac{2^{1+\varepsilon/2}}{\pi^{\varepsilon/2}} \right)^p h^{\varepsilon p} \sigma^{-\varepsilon p}.
$$
\end{proof}

\begin{lemma}
\label{lem:lina-1.3}
\[
\lim_{p \to \infty} \frac{R^*(t; \lambda)}{p}
= \lim_{p \to \infty} \sum_{m=1}^M \left( \frac{1}{p} \sum_{i=1}^p t_{R_m}(\hat{\beta}_i, \beta_i) \right)
= \sum_{m=1}^M \mathbb{E} \left[ t_{R_m} \left( \eta_{\alpha \tau}(\Pi + \tau Z), \Pi \right) \right].
\]
\end{lemma}

\begin{proof}[Proof of Lemma \ref{lem:lina-1.3}]
Note that $t_{R_m}(x,y)$ is not a pseudo-Lipschitz function, so we cannot apply Theorem 1.5 form 
\citet{bayati2012lasso}. 
To bypass the discontinuity of $t_{R_m}(x,y)$, define the  function 
\[ 
t_{R_{m,h}}(x,y) =
\begin{cases} 
1, & a_m(t) < x < b_m(t) \\ 
0, & x \geq b_m(t) \text{ or } x \leq a_m(t) \\ 
1 - \frac{x - b_m(t)}{h}, & b_m(t) < x < b_m(t)+h \\ 
1 + \frac{x - a_m(t)}{h}, & a_m(t)-h < x < a_m(t)
\end{cases},  
\]
which is pseudo-Lipschitz. 
It is clear that $t_{R_{m,h}}(x,y) \to t_{R_m}(x,y)$ as $h \to 0$, for fixed $x,y$. 
By Theorem 1.5 of 
\citet{bayati2012lasso}, we have 
\[
\lim_{p \to \infty} \frac{1}{p} \sum_{m=1}^M \sum_{i=1}^p t_{R_{m,h}}(\hat{\beta}_i, \beta_i) \xrightarrow{P} \sum_{m=1}^M \mathbb{E}\left[ t_{R_{m,h}}(\eta_{\alpha \tau}(\Pi + \tau Z), \Pi) \right].  
\]
Now, for all $\epsilon > 0$,
\[
\begin{aligned}
\mathbb{P}\left( \left| \frac{1}{p} \sum_{i=1}^p t_{R_{m,h}}(\hat{\beta}_i, \beta_i) - \frac{1}{p} \sum_{i=1}^p t_{R_m}(\hat{\beta}_i, \beta_i) \right| > \epsilon \right) &\leq \\ 
\mathbb{P}\left( \frac{1}{p} \sum_{i=1}^p \left| t_{R_{m,h}}(\hat{\beta}_i, \beta_i) - t_{R_m}(\hat{\beta}_i, \beta_i) \right| > \epsilon \right) &\leq \\ 
\mathbb{P}\left( \frac{1}{p} \sum_{i=1}^p \bm{1} \{b_m < \hat{\beta}_i < b_m + h\} > \epsilon \right) &+ \mathbb{P}\left(\frac{1}{p} \sum_{i=1}^p \bm{1} \{a_m - h < \hat{\beta}_i < a_m\} > \epsilon \right)
\end{aligned}
\]
If we show that:
\begin{equation} \label{proof lamma 1.3 1}
\lim_{h \to 0} \lim_{p \to \infty} \sup \mathbb{P}\left( \frac{1}{p} \sum_{i=1}^p \bm{1} \{b_m < \hat{\beta}_i < b_m + h\} > \epsilon \right) = 0
\end{equation} 
and
\begin{equation} \label{proof lamma 1.3 2}
\lim_{h \to 0} \lim_{p \to \infty} \sup \mathbb{P}\left( \frac{1}{p} \sum_{i=1}^p \bm{1} \{a_m - h < \hat{\beta}_i < a_m\} > \epsilon \right) = 0,  
\end{equation}
we get the following by invoking the dominated convergence theorem,  
\[ 
\begin{aligned}
\lim_{p \to \infty} \frac{1}{p} \sum_{m=1}^M \sum_{i=1}^p t_{R_m}(\hat{\beta}_i, \beta_i) &\xrightarrow{P} \sum_{m=1}^M \lim_{h \to 0} \mathbb{E}\left[ t_{R_{m,h}}(\eta_{\alpha \tau}(\Pi + \tau Z), \Pi) \right] \\
&= \sum_{m=1}^M \mathbb{E}\left[ t_{R_m}(\eta_{\alpha \tau}(\Pi + \tau Z), \Pi) \right]. 
\end{aligned}
\]
To prove (\ref{proof lamma 1.3 1}), let $A$ be the active set of the Lasso solution. 
From the KKT conditions, 
\[ 
\boldsymbol{X}_A^T(Y - \boldsymbol{X}_A \hat{\beta}_A) = \tilde{\lambda}_A \quad \text{with } \tilde{\lambda}_A = \lambda \cdot \text{sign}(\hat{\beta}_A). 
\]
Since $|A| \leq n$, $\boldsymbol{X}_A^T \boldsymbol{X}_A$ is invertible, and thus
\[
\hat{\beta}_A = (\boldsymbol{X}_A^T\boldsymbol{X}_A)^{-1}(\boldsymbol{X}_A^T Y - \tilde{\lambda}_A). 
\]
Now, for any subset $D \subset \{1, 2, \dots, p\}$ with $|D| \leq n$, and any $\tilde{\lambda}$ of length $|D|$ with each component $\pm \lambda$, define
\[
\hat{\beta}_D^{\tilde{\lambda}} = (\boldsymbol{X}_D^T \boldsymbol{X}_D)^{-1}(\boldsymbol{X}_D^T \boldsymbol{X}\beta - \tilde{\lambda}) + (\boldsymbol{X}_D^T \boldsymbol{X}_D)^{-1} \boldsymbol{X}_D^T \zeta. 
\]
According to Lemma 3 of \citet{bogdan2013supplement}
, there exists $\rho > 0$ such that $\mathbb{P}(|A| < \rho p) = o(1)$. 
For any $\epsilon', \epsilon > 0$, for large $n, p$, we have 
\[
\mathbb{P}\left(\frac{1}{p} \sum_{i=1}^p \bm{1}\{b_m < \hat{\beta}_i < b_m + h \} > \epsilon \right)
\leq 
\sum_{|D| = \rho p}^{{min}(p,n)} 
\sum_{|\tilde{\lambda}| = |D|} \Bigg[\]
\[
\quad \mathbb{P}\left( \frac{1}{p} \sum_{i=1}^{|D|} \bm{1}\{b_m < \hat{\beta}_D^{\tilde{\lambda}}(i) < b_m + h\},\ \sigma_{\max}(\boldsymbol{X}) < 1 + \\
\delta^{-1/2} + \epsilon' \right)
+\ o(1)\ +\ \mathbb{P}(|A| < \rho p) \Bigg]
\]
\smallskip
Since $\hat{\beta}_D^{\tilde{\lambda}}$ is multivariate normal with bounded covariance under these conditions, applying Lemma~\ref{lem:lina-1.2} gives
\[
\mathbb{P}\left(\frac{1}{p} \sum_{i=1}^{|D|} \mathbf{1} \{b_m < \hat{\beta}_D^{\tilde{\lambda}}(i) < b_m + h \} > \epsilon \right) \leq (c h^\epsilon \sigma^{-\epsilon})^{\rho p}. 
\]
Hence, 
\[
\mathbb{P}\left(\frac{1}{p} \sum_{i=1}^p \bm{1}\{b_m < \hat{\beta}_i < b_m + h\}> \epsilon\right) \leq (4c^\rho h^{\rho \epsilon} \sigma^{-\rho \epsilon})^p + o(1) + o(1). 
\]
This implies
\[
\lim_{h \to 0} \lim_{p \to \infty} \sup \mathbb{P}\left( \frac{1}{p} \sum_{i=1}^p \bm{1}\{b_m < \hat{\beta}_i < b_m + h\} > \epsilon \right) = 0
\]
Similarly, we can prove \eqref{proof lamma 1.3 2}, which leads to 
\[
\lim_{p \to \infty} \frac{R^*(t; \lambda)}{p} = \sum_{m=1}^M \mathbb{E}\left[ t_{R_m}(\eta_{\alpha \tau}(\Pi + \tau Z), \Pi) \right].
\]
\end{proof}

\smallskip
\begin{lemma}
\label{lem:lina-1.4}
    \[
\lim_{p \to \infty} \frac{V^*(t; \lambda)}{p} = \lim_{p \to \infty} \sum_{m=1}^M \left(\frac{1}{p} \sum_{i=1}^{p} t_{V_m}(\hat{\beta}_i, \beta_i)\right) \xrightarrow{P} \sum_{m=1}^M \mathbb{E}[t_{V_m}(\eta_{\alpha \tau}(\Pi + \tau Z), \Pi)].
\]
\end{lemma}

\begin{proof}[Proof of Lemma~\ref{lem:lina-1.4}]
The proof is similar to the proof of Lemma~\ref{lem:lina-1.3}
\end{proof}

\begin{proof}[Proof of Lemma~\ref{lem:lina-1.1}]
From Lemmas~\ref{lem:lina-1.3}
and \ref{lem:lina-1.4} 
we get
\[
\text{FDP}^*(t; \lambda) = \frac{V^*(t; \lambda)}{R^*(t; 
\lambda) \vee 1}
\overset{P}{\longrightarrow}
\frac{\mathbb{E}\left[\sum_{m=1}^M t_{V_m}(\eta_{\alpha\tau}(\Pi + \tau Z), \Pi)\right]}
     {\mathbb{E}\left[\sum_{m=1}^M t_{R_m}(\eta_{\alpha\tau}(\Pi + \tau Z), \Pi)\right]}\]
     \[\hspace{5cm} =\frac{\sum_{m=1}^M \mathbb{E}\left[t_{V_m}(\eta_{\alpha\tau}(\Pi + \tau Z), \Pi)\right]}
     {\sum_{m=1}^M \mathbb{E}\left[t_{R_m}(\eta_{\alpha\tau}(\Pi + \tau Z), \Pi)\right]}.
\]
\end{proof}

\smallskip

\begin{proof}[Proof of Theorem \ref{thm:predictions-ol}]

First we prove statement \eqref{eq:fdp-oracle}. 
By Assumption~\ref{assmptn:1}, we have 
\begin{align}\label{Proof TH. 3.1 1 Assumption}
& \frac{q_0(x)}{q(x)} \leq t  \iff  x\in \bigcup_{m=1}^M (a_m(t), b_m(t)) ,
\end{align}
where the intervals $(a_m(t), b_m(t))$ are disjoint. Then we get that
\begin{align*}
&\text{FDP}^*(t; \lambda) = \frac{\sum_{i=1}^p \bm{1}\left\{ \hat{\beta}_i \in \bigcup_{m=1}^M (a_m(t), b_m(t)), \beta_i = 0 \right\}}{\sum_{i=1}^p \bm{1}\left\{ \hat{\beta}_i \in \bigcup_{m=1}^M (a_m(t), b_m(t)) \right\}}.
\end{align*}
We need to show that
\begin{align*}
&\text{FDP}^*(t; \lambda) \xrightarrow{P} \mathbb{P}\left( \Pi = 0 \mid  \frac{q_0(\eta_{\alpha \tau}(\Pi + \tau Z))}{q(\eta_{\alpha \tau}(\Pi + \tau Z))} \leq t \right).
\end{align*}
Note that
\begin{align*}
&\text{FDP}^*(t; \lambda) = \frac{\sum_{m=1}^M \sum_{i=1}^p \bm{1}\left\{ \hat{\beta}_i \in (a_m(t), b_m(t)), \beta_i = 0 \right\}}{\sum_{m=1}^M \sum_{i=1}^p \bm{1}\left\{\hat{\beta}_i \in (a_m(t), b_m(t)) \right\}}.
\end{align*}
We can write this as
\begin{align*}
&\frac{\sum_{m=1}^M \frac{1}{p} \sum_{i=1}^p \bm{1}\{  \hat{\beta}_i \in (a_m(t), b_m(t)), \beta_i = 0 \}}{\sum_{m=1}^M \frac{1}{p} \sum_{i=1}^p \bm{1}\{  \hat{\beta}_i \in (a_m(t), b_m(t)) \}}\\[2ex]
&= \text{FDP}^{\text{OR}}(t; \lambda)
\end{align*}
which, by Lemma~\ref{lem:lina-1.1},  
\begin{align*}
&\xrightarrow{P} \frac{\sum_{m=1}^M \mathbb{E}\left[ \mathbf{1}\{ \eta_{\alpha \tau}(\Pi + \tau Z) \in (a_m(t), b_m(t)), \Pi = 0 \} \right]}{\sum_{m=1}^M \mathbb{E}\left[ \mathbf{1}\{\eta_{\alpha \tau}(\Pi + \tau Z) \in (a_m(t), b_m(t)) \} \right]}\\[2ex]
&= \frac{\mathbb{E}\left[ \sum_{m=1}^M \mathbf{1}\{ \eta_{\alpha \tau}(\Pi + \tau Z) \in (a_m(t), b_m(t)), \Pi = 0 \} \right]}{\mathbb{E}\left[ \sum_{m=1}^M \mathbf{1}\{ \eta_{\alpha \tau}(\Pi + \tau Z) \in (a_m(t), b_m(t)) \} \right]}\\
&= \frac{\mathbb{E}\left[ \mathbf{1}\left\{ \eta_{\alpha \tau}(\Pi + \tau Z) \in \bigcup_{m=1}^M (a_m(t), b_m(t)), \Pi = 0 \right\} \right]}{\mathbb{E}\left[ \mathbf{1}\left\{\eta_{\alpha \tau}(\Pi + \tau Z) \in \bigcup_{m=1}^M (a_m(t), b_m(t)) \right\} \right]},
\end{align*}
where the last equality uses the fact that the intervals in the union are disjoint. 
By the assumption \eqref{Proof TH. 3.1 1 Assumption}, we get the following equality

\begin{align*}
&= \frac{\mathbb{E}\left[ \mathbf{1}\left\{ \frac{q_0(\eta_{\alpha \tau}(\Pi + \tau Z))}{q(\eta_{\alpha \tau}(\Pi + \tau Z))} \leq t, \Pi = 0 \right\} \right]}{\mathbb{E}\left[ \mathbf{1}\left\{ \frac{q_0(\eta_{\alpha \tau}(\Pi + \tau Z))}{q(\eta_{\alpha \tau}(\Pi + \tau Z))} \leq t \right\} \right]}\\[2ex]
&= \mathbb{P}\left( \Pi = 0 \mid \frac{q_0(\eta_{\alpha \tau}(\Pi + \tau Z))}{q(\eta_{\alpha \tau}(\Pi + \tau Z))} \leq t \right)\\
&\equiv \operatorname{fdp}^*(t; \lambda).
\end{align*}
\noindent We turn to the proof of statement \eqref{eq:tpp-oracle}. 
By Lemma~\ref{lem:lina-1.3} and Lemma~\ref{lem:lina-1.4} and assumption \eqref{Proof TH. 3.1 1 Assumption}, we have 
\begin{align*}
\text{TPP}^*(t; \lambda) &= \frac{\frac{1}{p}(R^*(t; \lambda) - V^*(t; \lambda))}{\frac{1}{p} \sum_{i=1}^{p} \mathbf{1}(\beta_i \neq 0)} \\[2ex]
&\xrightarrow{P} \frac{{\sum_{m=1}^M \mathbb{E}\left[t_{R_m}(\eta_{\alpha\tau}(\Pi + \tau Z), \Pi)\right]}-
     {\sum_{m=1}^M \mathbb{E}\left[t_{V_m}(\eta_{\alpha\tau}(\Pi + \tau Z), \Pi)\right]}}
{\mathbb{E} \left[ \mathbf{1}(\Pi \neq 0) \right]}\\[2ex]
&= \frac{\mathbb{E}\left[ \mathbf{1}\left\{ \frac{q_0(\eta_{\alpha \tau}(\Pi + \tau Z))}{q(\eta_{\alpha \tau}(\Pi + \tau Z))} \leq t \right\} - \mathbf{1}\left\{ \frac{q_0(\eta_{\alpha \tau}(\Pi + \tau Z))}{q(\eta_{\alpha \tau}(\pi + \tau Z))} \leq t, \Pi = 0 \right\} \right] }{\mathbb{E} \left[ \mathbf{1}(\Pi \neq 0) \right]}\\[2ex]
&= \frac{{\mathbb{E}\left[ \mathbf{1}\left\{ \frac{q_0(\eta_{\alpha \tau}(\Pi + \tau Z))}{q(\eta_{\alpha \tau}(\Pi + \tau Z))} \leq t, \Pi \neq 0 \right\} \right]}}{\mathbb{E} \left[ \mathbf{1}(\Pi \neq 0) \right]}\\[2ex]
&= \mathbb{P}\Big( \frac{q_0(\eta_{\alpha \tau}(\Pi + \tau Z))}{q(\eta_{\alpha \tau}(\Pi + \tau Z))} \leq t\mid \Pi \neq 0\Big)\\[2ex]
&= \mathbb{P}\Big( \frac{q_0(\eta_{\alpha \tau}(\Pi_1 + \tau Z))}{q(\eta_{\alpha \tau}(\Pi_1 + \tau Z))} \leq t\Big)\\[2ex]
&\equiv \operatorname{tpp}^*(t; \lambda).
\end{align*}
\end{proof}

\begin{proof}[Proof of Theorem \ref{thm:phi-optimal}]

Under Assumption \ref{assmptn:1} and the conditions assumed on $\varphi$, the problem reduces to showing the following. 
Let $X = \eta_{\alpha\tau}(\Pi + \tau Z)$ where $\Pi$ is as in \eqref{eq:iid-two-groups}, and consider testing $H_0: \Pi = 0$ against $H_1: \Pi\neq 0$. 
Since 
$$\frac{q_0(x)}{q(x)} = \frac{q_0(x)}{(1-\epsilon) q_0(x) + \epsilon q_1(x)}, 
$$
rejecting for $\frac{q_0(x)}{q(x)} <t^*$ is equivalent to rejecting when  $\frac{q_1(x)}{q_0(x)} >s(t^*)$ for a corresponding cutoff $s(t^*)$. 
Thus, define the (likelihood ratio) test 
\[
\mathcal R_\lambda^*(x)
=
\mathbf 1\left\{\frac{q_1(x)}{q_0(x)} > s^* \right\},
\]
where $q_0$ is the unnormalized density of $\eta_{\alpha\tau}(\tau Z) | |Z| > \alpha$, $q_1$ is the unnormalized density of $\eta_{\alpha\tau}(\Pi_1 + \tau Z) | |\Pi_1 + \tau Z| > \alpha\tau$, and, to simplify notation, we write $s^* =: s(t^*)$. 
Now define also the test 
$$
\mathcal R_\lambda^\varphi(x) = \mathbf 1 \left\{ \varphi(x) > t\right\}. 
$$
for some other $\varphi$. 
We need to show that if the constants $s^*$ and $t$ are such that 
$$
\PP(\mathcal R_\lambda^*(X) =1 \mid \Pi\neq 0) = 
\PP(\mathcal R_\lambda^\varphi(X) =1 \mid \Pi\neq 0)=\xi, 
$$
then 
\begin{equation}
\label{eq:type-1-posterior}
\PP(\Pi=0\lvert \mathcal{R}_\lambda^*(X)=1) \leq \PP(\Pi=0\lvert \mathcal{R}_\lambda^\varphi(X)=1).     
\end{equation}
Since the power is fixed at $\xi$, proving \eqref{eq:type-1-posterior} is equivalent to proving that the type I errors in the frequentist (fixed $\Pi$) testing problem satisfy 
$$
\PP(\mathcal{R}_\lambda^*(X)=1 \lvert \Pi=0 ) \leq \PP(\mathcal{R}_\lambda^\varphi(X)=1 \lvert \Pi=0 ). 
$$
It therefore suffices to adapt the proof of the Neyman-Pearson lemma so that instead of showing that the LR test has maximum power among all test with the same type I error probability, we show that it minimizes the probability of a type I error among all tests with the same power. 

Consider the integral
\[
\int \bigl(\mathcal{R}_\lambda^*(x)-\mathcal{R}_\lambda^\varphi (x)\bigr)
\bigl(q_1(x)-s^* \ q_0(x)\bigr)\,dx.
\]
By the construction of $\mathcal{R}_\lambda^*(x)$, $\mathcal R_\lambda^*(x)=1$ whenever
$q_1(x)-s^* \ q_0(x)>0$ and $\mathcal{R}_\lambda^*(x)=0$ whenever
$q_1(x)-s^* \ q_0(x)<0$. Since $\mathcal{R}_\lambda^\varphi(x)\in\{0,1\}$, we get that the integral is nonnegative.\\[0.5ex]
Expanding the integral yields
\begin{align*}
0
&\le
\int \bigl(\mathcal{R}_\lambda^*(x)-\mathcal{R}_\lambda(x)\bigr) q_1(x)\,dx
-
s^*\int \bigl(\mathcal{R}_\lambda^*(x)-\mathcal{R}_\lambda^\varphi(x)\bigr) q_0(x)\,dx \\
&=
\mathbb{P}(\mathcal{R}_\lambda^*(X) =1 \mid \Pi\neq 0)
-
\mathbb{P}(\mathcal{R}_\lambda^\varphi(X) =1 \mid \Pi\neq 0) \\
&\quad
-
s^*\bigl(
\mathbb{P}(\mathcal{R}_\lambda^*(X) =1 \mid \Pi=0)
-
\mathbb{P}(\mathcal{R}_\lambda^\varphi(X) =1 \mid \Pi=0)
\bigr).
\end{align*}
Using the assumption that both tests have the same power $\xi$, we obtain
\[
\mathbb{P}(\mathcal{R}_\lambda^*(X) =1 \mid \Pi=0)
\leq
\mathbb{P}(\mathcal{R}_\lambda^\varphi(X) =1 \mid \Pi=0).
\]
\end{proof}

To establish Theorem~\ref{thm:predictions-eb}, it is essential to demonstrate that the estimator \( \hat{q}(x) \) converges 
uniformly to the true density \( q(x) \). To this end, we begin by proving the pointwise 
convergence in probability of \( \hat{q}(x) \) to \( q(x) \). Lemma~\ref{lem:lina-2.1.1}, proved below, provides the foundation for this intermediate result.




\begin{proof}[Proof of Lemma \ref{lem:lina-2.1.1}]

Let $\tilde{\psi}(x, y) = \mathbf{1}\{x \ne 0\} \, \psi(x, y)$, we get that
\[
\frac{1}{p} \sum_{i=1}^p \mathbf{1}\{{\hat{\beta}_i \ne 0}\} \, \psi(\hat{\beta}_i, \beta_i) = \frac{1}{p} \sum_{i=1}^p \tilde{\psi}(\hat{\beta}_i, \beta_i).
\]
Note that $\tilde{\psi}$ is not pseudo-Lipschitz, so we cannot apply Theorem \ref{thm:manual} of \citet{bayati2012lasso}.  
We will show that, although this theorem does not apply directly, Lemma~\ref{lem:lina-2.1.1} is still correct. 
Thus, define a pseudo-Lipschitz function
\[
\tilde{\psi}_{t'}(x, y) = \left(1 - Q\left(\frac{x}{t'}\right)\right) \psi(x, y),
\]
where $Q(x) = \max(1 - |x|, 0)$ and $t' > 0$. 
Because $\tilde{\psi}_{t'}$ is pseudo-Lipschitz, by Theorem \ref{thm:manual} of \citet{bayati2012lasso} we have 
\[
\lim_{p \to \infty} \frac{1}{p} \sum_{i=1}^p \tilde{\psi}_{t'}(\hat{\beta}_i, \beta_i) = \mathbb{E}[\tilde{\psi}_{t'}(\eta_{\alpha \tau}(\Pi + \tau Z), \Pi)].
\]
Now, 
\begin{align*}
&\left| \tilde{\psi}_{t'}(x, y) - \tilde{\psi}(x, y) \right| \\[2ex]
&= \left| \left(1 - Q\left(\frac{x}{{t'}}\right) - \mathbf{1}\{x \ne 0\} \right) \psi(x, y) \right| \\[2ex]
&\leq \mathbf{1} \left\{0 < |x| < {t'} \right\} |\psi(x, y)|\\[2ex]
&\leq \mathbf{1}\left\{0 < |x| < {t'} \right\} \sup_{x \in \mathbb{R}} |\psi(x, y)|\\[2ex]
& = \mathbf{1}\left\{0 < |x| < {t'} \right\} \|\psi\|_\infty.
\end{align*}
We now bound the probability:
\begin{align*}
&\mathbb{P} \left( \left| \frac{1}{p} \sum_{i=1}^p \tilde{\psi}_{t'}(\hat{\beta}_i, \beta_i) - \frac{1}{p} \sum_{i=1}^p \tilde{\psi}(\hat{\beta}_i, \beta_i) \right| > \epsilon \right) \\
&= \mathbb{P} \left( \frac{1}{p} \left| \sum_{i=1}^p \left( \tilde{\psi}_{t'}(\hat{\beta}_i, \beta_i) - \tilde{\psi}(\hat{\beta}_i, \beta_i) \right) \right| > \epsilon \right) \\
&\leq \mathbb{P} \left( \frac{1}{p} \sum_{i=1}^p \left| \tilde{\psi}_{t'}(\hat{\beta}_i, \beta_i) - \tilde{\psi}(\hat{\beta}_i, \beta_i) \right| > \epsilon \right) \\
&\leq \mathbb{P} \left( \frac{1}{p} \sum_{i=1}^p \mathbf{1}\left\{0 < |\hat{\beta}_i| < {t'} \right\}\, \|\psi\|_\infty > \epsilon \right) \\
&= \mathbb{P} \left( \frac{1}{p} \sum_{i=1}^p \mathbf{1}\left\{0 < |\hat{\beta}_i| < {t'} \right\} > \tilde{\epsilon} \right), \quad \quad \quad \tilde{\epsilon} = \frac{\epsilon}{\|\psi\|_\infty}.
\end{align*}
\\
In \citet{bogdan2013supplement} it is shown that 
\[
\lim_{{t'} \to 0} \lim_{p \to \infty} \sup \mathbb{P} \left( \frac{1}{p} \sum_{i=1}^p \mathbf{1}\left\{0 < |\hat{\beta_i}| < {t'} \right\} > \tilde{\epsilon} \right) = 0,
\]
hence, for any $\tilde{\epsilon} > 0$, 
\begin{align*}
\lim_{p \to \infty} \frac{1}{p} \sum_{i=1}^p \tilde{\psi}(\hat{\beta}_i, \beta_i)
&= \lim_{{t'} \to 0} \lim_{p \to \infty} \frac{1}{p} \sum_{i=1}^p \tilde{\psi}_{t'}(\hat{\beta}_i, \beta_i)\\[2ex]
&= \lim_{{t'} \to 0} \mathbb{E}\left[\tilde{\psi}_{t'}(\eta_{\alpha \tau}(\Pi + \tau Z), \Pi)\right]\\[2ex]
&= \mathbb{E}[\tilde{\psi}(\eta_{\alpha \tau}(\Pi + \tau Z), \Pi)],
\end{align*}
where the last equality follows from applying the dominated convergence theorem to $\tilde{\psi}_{t'} \to \tilde{\psi}$.
\end{proof}

\begin{proof}[Proof of Proposition~\ref{prop:lina-2.1}]
Let
\[
k(\hat{\beta}_i, x) = \frac{1}{h} \phi\left(\frac{x - \hat{\beta}_i}{h} \right) = \frac{1}{\sqrt{2\pi h^2}} \exp\left( -\frac{1}{2} \cdot \frac{(\hat{\beta}_i - x)^2}{h^2} \right), \quad h > 0 \ (\text{fixed}),\  i = 1,..., p.
\]
Define:
\[
\psi(\hat{\beta}_i, \beta_i) = k(\hat{\beta}_i, x), \quad \text{where } x \neq0 \text{ is fixed}.
\]
Then $\psi: \mathbb{R} \times \mathbb{R} \to \mathbb{R}$ is a pseudo-Lipschitz function. 
Now define
\[
\tilde{\psi}(\hat{\beta}_i, \beta_i) = \mathbf{1} \left\{\hat{\beta}_i\ne 0\right\} \cdot \psi(\hat{\beta}_i, \beta_i).
\]
We then get
\[
\frac{1}{p} \sum_{i=1}^p \mathbf{1}\left\{\hat{\beta}_i\ne 0\right\}\cdot \psi(\hat{\beta}_i, \beta_i) = \frac{1}{p} \sum_{i=1}^p \tilde{\psi}(\hat{\beta}_i, \beta_i).
\]
Define 
\[
\hat{q}(x) = \hat{q}_h(x) = \frac{1}{p} \sum_{i=1}^{p} \mathbf{1}\left\{\hat{\beta}_i\ne 0\right\} \cdot \frac{1}{h} \phi\left(\frac{x - \hat{\beta}_i}{h} \right) 
\]
So, 
\begin{align*}
\hat{q}(x) &= \frac{1}{p} \sum_{i=1}^{p} \mathbf{1}\left\{\hat{\beta}_i\ne 0\right\} \cdot k(\hat{\beta}_i, x)\\[2ex]
&= \frac{1}{p} \sum_{i=1}^{p} \mathbf{1}\left\{\hat{\beta}_i\ne 0\right\} \cdot \psi(\hat{\beta}_i, \beta_i)\\[2ex]
&\xrightarrow{P} \mathbb{E}[\tilde{\psi}(\eta_{\alpha \tau}(\Pi + \tau Z), \Pi)]  \quad \text{(by Lemma~\ref{lem:lina-2.1.1})} \\[2ex]
&= \mathbb{E} \left[\mathbf{1}\left\{\eta_{\alpha \tau}(\Pi + \tau Z) \neq 0\right\} \cdot \psi\left(\eta_{\alpha \tau}(\Pi + \tau Z), \Pi\right)\right]
\end{align*}
Now consider 
\[
\mathbb{E} \left[\mathbf{1}\left\{\eta_{\alpha \tau}(\Pi + \tau Z) \neq 0\right\} \cdot \psi\left(\eta_{\alpha \tau}(\Pi + \tau Z),\Pi\right)\right]. 
\]
Let \( Y \sim \eta_{\alpha \tau}(\Pi + \tau Z) \) with density \(f(y)= (1 - w) \delta_0 (y) + w g(y)\) = $(1 - w) \delta_0 (y) + q(y)$. 
Then 
\begin{align*}
&\mathbb{E} \left[\mathbf{1}\left\{\eta_{\alpha \tau}(\Pi + \tau Z) \neq 0\right\} \cdot \psi\left(\eta_{\alpha \tau}(\Pi + \tau Z),\Pi\right)\right]  \\[2ex]
&=\int_{-\infty}^{\infty} \mathbf{1}(y \neq 0) \cdot \frac{1}{h} \phi\left(\frac{x - y}{h}\right) f(y) \, dy \\[2ex]
&= \int_{-\infty}^{\infty} \frac{1}{h} \phi\left(\frac{x - y}{h}\right) q(y) \, dy \\[2ex]
&= \int_{-\infty}^{0} \frac{1}{h} \phi\left(\frac{x - y}{h}\right) q(y) \, dy \ +
\int_{0}^{\infty} \frac{1}{h} \phi\left(\frac{x - y}{h}\right) q(y) \, dy
\end{align*}
Changing variables by letting \( u = \frac{x - y}{h} \Rightarrow y = x - uh \), we can continue 
\[
= \int_{-\infty}^{\frac{x}{h}} \phi(u) \cdot q(x - uh) \, du + \int_{\frac{x}{h}}^{\infty} \phi(u) \cdot q(x - uh) \, du
\]
Now by Taylor expansion of \( q(x - uh) \) around \( x \), assuming \( q \) has at least one derivative, 
\[
q(x - hu) = q(x) - hu \cdot q'(x) + o(h).
\]
Since
 \begin{equation*}
     \lim_{h \to 0^+} q(x - hu) = q(x) \ and \lim_{h \to 0^-} q(x - hu) = q(x), 
 \end{equation*}
we get, 
\begin{equation*}
    \lim_{h \to 0} q(x - hu) = q(x)
\end{equation*}
Suppose also that \( \underset{{x \in \mathbb{R}\setminus \{0\}}}{\sup} |q(x)| = M < \infty \), then by the dominated convergence theorem, 
$$
\begin{aligned}
\lim_{h \to 0}  \lim_{p \to \infty} \hat{q}(x) &= 
\lim_{h \to 0} \left( \lim_{p \to \infty} \frac{1}{p} \sum_{i=1}^{p} \mathbf{1}(\hat{\beta}_i \neq 0) \cdot \frac{1}{h} \phi\left(\frac{x - \hat{\beta}_i}{h} \right) \right)\\
&= \lim_{h \to 0} \mathbb{E} \left[ \mathbf{1}\left\{ \eta_{\alpha \tau}(\Pi + \tau Z) \neq 0 \right\} \cdot \psi\left( \eta_{\alpha \tau}(\Pi + \tau Z), \Pi \right) \right]  \quad \text{(by Lemma~\ref{lem:lina-2.1.1})}\\
&= q(x), 
\end{aligned}
$$
where we used the fact that $\int \phi(u)  \  du = 1$. 
\end{proof}
\medskip
\medskip

\medskip

\begin{proof}[Proof of Proposition \ref{prop:lina-2.2.a}]
We want to show that 
\[
\sup_{x \in [\Delta, \mathcal{L}]} \left| \hat{q}(x) - q(x) \right| \xrightarrow{P} 0.
\]
Define equally spaced points \( \Delta = x_0 < x_1 < x_2 < \cdots < x_{l} = \mathcal{L} \), with 
\[
\Delta' = x_{i+1} - x_i = \frac{1}{l}, \quad i = 0, 1, \ldots, l-1,
\]
where \( l \) will be specified below.
\\
Since by Assumption~\ref{assmptn:2} \( q(x) \) is uniformly continuous on \([ \Delta, \mathcal{L} ]\), for any \( \epsilon' > 0 \), there exists \( l \) sufficiently large such that
\begin{align}\label{proof proposition 2.2.a 1}
|q(x) - q(x_i)| < \frac{\epsilon'}{2}
\end{align}
for all \( x, x_i \in [\Delta, \mathcal{L}] \) satisfying \( |x - x_i| \leq \Delta' = \frac{1}{l} \).
\\
By Proposition~\ref{prop:lina-2.1}, we know that \( \hat{q}(x) \xrightarrow{P} q(x) \) pointwise for fixed \( x \in [\Delta, \mathcal{L}] \), i.e.,
\[
|\hat{q}(x) - q(x)| \xrightarrow{P} 0.
\]
By a union bound, we have
\begin{align}\label{proof proposition 2.2.a 2}
\max_{0 \leq i \leq l} |\hat{q}(x_i) - q(x_i)| \xrightarrow{P} 0.
\end{align}
That is, for any \( \epsilon > 0 \),
\begin{align*}
    &\mathbb{P}\left( \underset{{0 \leq i \leq l}}{\max}|\hat{q}(x_i) - q(x_i)| > \epsilon \right) \\[2ex]
    &= \mathbb{P}\left( \underset{{0 \leq i \leq l}}{\bigcup}  \left\{|\hat{q}(x_i) - q(x_i)| > \epsilon \right\} \right)\\[2ex]
    &\leq \sum_{i=0}^{l} \mathbb{P}\left( |\hat{q}(x_i) - q(x_i)| > \epsilon \right) \xrightarrow{P} 0.
\end{align*}

\smallskip
\smallskip

\noindent Now we want to show that for sufficiently large (but fixed) \( l \),
\[
\max_{0 \leq i < l} \sup_{x_i \leq x \leq x_{i+1}} \left| \hat{q}(x) - \hat{q}(x_i) \right| \leq \frac{\epsilon'}{2}.
\]
For any \( x \in [\Delta, \mathcal{L}] \), there exists \( i \) such that \( x \in [x_i, x_{i+1}] \), and so \( |x - x_i| \leq \Delta' = \frac{1}{l} \). Then,

\begin{align}\label{proof proposition 2.2.a 3}
|\hat{q}(x) - \hat{q}(x_i)|
&= \left| \frac{1}{ph} \sum_{j=1}^p \phi\left( \frac{x - \hat{\beta}_j}{h} \right) \mathbf{1}\{\hat{\beta}_j \neq 0\}
- \frac{1}{ph} \sum_{j=1}^p \phi\left( \frac{x_i - \hat{\beta}_j}{h} \right) \mathbf{1}\{\hat{\beta}_j \neq 0\} \right| \notag\\
&= \left| \frac{1}{ph} \sum_{j=1}^p \left[\phi\left( \frac{x - \hat{\beta}_j}{h} \right) - \phi\left( \frac{x_i - \hat{\beta}_j}{h} \right) \right] \mathbf{1}\{\hat{\beta}_j \neq 0\} \right| \notag \\
&\leq \frac{1}{ph} \sum_{j=1}^p \left| \phi\left( \frac{x - \hat{\beta}_j}{h} \right) - \phi\left( \frac{x_i - \hat{\beta}_j}{h} \right) \right| \notag \\
&\leq \frac{1}{ph} \sum_{j=1}^p L \frac{|x - x_i|}{h} \leq \frac{L \Delta'}{h^2}, 
\end{align}
where \( L \) is a Lipschitz constant for the continuous Lipschitz kernel function \( \phi \).
\\
Now for any \( x \in [\Delta, \mathcal{L}] \), there exists \( i \) such that \( x \in [x_i, x_{i+1}] \), and we have 
\begin{align}\label{proof proposition 2.2.a 4}
|\hat{q}(x) - q(x)| 
&= |\hat{q}(x) - \hat{q}(x_i) + \hat{q}(x_i) - q(x_i) + q(x_i) - q(x)| \notag \\
&\leq |\hat{q}(x) - \hat{q}(x_i)| + |\hat{q}(x_i) - q(x_i)| + |q(x_i) - q(x)| \notag\\
&\overset{\eqref{proof proposition 2.2.a 1} , \eqref{proof proposition 2.2.a 3}}{\leq }\frac{L \Delta'}{h^2} + \max_{0 \leq i \leq l} |\hat{q}(x_i) - q(x_i)| + \frac{\epsilon'}{2}. 
\end{align}
\\
If we choose \( l \) sufficiently large such that 
\[
\Delta' = \frac{1}{l} \leq \frac{\epsilon' h^2}{2L},
\]
then the term \( \frac{L \Delta'}{h^2} \leq \frac{\epsilon'}{2} \). Hence, the bound in \eqref{proof proposition 2.2.a 4} becomes
\begin{equation}
\label{proof proposition 2.2.a 5, prop 4.5}
   |\hat{q}(x) - q(x)| \leq \frac{\epsilon'}{2} + \max_{0 \leq i \leq l} |\hat{q}(x_i) - q(x_i)| + \frac{\epsilon'}{2}. 
\end{equation}
By \eqref{proof proposition 2.2.a 1}, \eqref{proof proposition 2.2.a 2}, and \eqref{proof proposition 2.2.a 5, prop 4.5}, and since \( \epsilon' > 0 \) is arbitrary small, we conclude that
\[
\sup_{x \in [\Delta, \mathcal{L}]} \left| \hat{q}(x) - q(x) \right| \xrightarrow{P} 0.
\]
\end{proof}





\begin{lemma}
\label{lem:lina-2.2}
    Suppose \( \hat{q}\) converges uniformly in probability to \( q \) on $[-\mathcal{L}, -\Delta]$ and on $[\Delta, \mathcal{L}]$.
Then $\hat{q}\rightarrow q$ uniformaly in probability on $[-\mathcal{L}, -\Delta] \cup [\Delta, \mathcal{L}]$. 
\end{lemma}

\begin{proof}[Proof of~\ref{lem:lina-2.2}]

Since $\hat{q} \to q $ uniformly in probability on $[-\mathcal{L}, -\Delta] $ and on $[\Delta, \mathcal{L}]$ , we find that for every $\epsilon > 0$, there exists $P_1$ such that for all $p \geq P_1$,
\[
\sup_{x \in [-\mathcal{L}, -\Delta]} \left| \hat{q}(x) - q(x) \right| < \epsilon
\]
with high probability. 
Similarly, for every $\epsilon > 0$, there exists $P_2$ such that for all $p \geq P_2$,
\[
\sup_{x \in[\Delta, \mathcal{L}]} \left| \hat{q}(x) - q(x) \right| < \epsilon
\]
with high probability.
\\
Let $N = \max(P_1, P_2)$. Then for all $p \geq N$ and every $\epsilon > 0$, we have
\[
\sup_{x \in [-\mathcal{L}, -\Delta] \cup [\Delta, \mathcal{L}]} \left| \hat{q}(x) - q(x) \right| < \epsilon
\]
with high probability. 
Alternatively, we can express this as
\[
\mathbb{P} \left( \sup_{x \in [-\mathcal{L}, -\Delta] \cup [\Delta, \mathcal{L}]} \left| \hat{q}(x) - q(x) \right| > \epsilon \right) < \delta,
\quad \text{for all small } \delta > 0.
\]

\end{proof}

\begin{proposition}
\label{prop: lina:2.3}
    $\forall \epsilon > 0, \quad \mathbb{P}\left( \underset{x \in I}{\sup} \left| 
\hat{q}_0(x) - q_0(x) \right| > \epsilon \right) \to 0,$
where \( I = [-\mathcal{L}, -\Delta] \cup [\Delta, \mathcal{L}] \), and $\mathcal{L},\  \Delta$ are defined in Assumptions \eqref{assmptn:1} and \eqref{assmptn:2}.
\end{proposition}

\begin{proof}[Proof of Proposition~\ref{prop: lina:2.3} ]
\[
\sup_{x \in I} \left| {\hat{q}_0(x)} - q_0(x) \right|
 \xrightarrow{P} 0,
\]
where
\[
q_0(x) = \frac{1}{\tau} \phi \left( \frac{x + \alpha\tau \cdot \text{sign}(x)}{\tau} \right), \quad
\hat{q}_0(x) = \frac{1}{\hat{\tau}} \phi \left( \frac{x + \widehat{\alpha\tau} \cdot \text{sign}(x)}{\hat{\tau}} \right),
\]
\( \phi(x) = \frac{1}{\sqrt{2\pi}} e^{-x^2/2} \) is the standard normal density, and
\( \hat{\tau} \xrightarrow{P} \tau \), \( \widehat{\alpha\tau} \xrightarrow{P} \alpha\tau \).
\\
Let \( \epsilon > 0 \). Since \( \phi \) is uniformly continuous on any compact interval and \( [-\mathcal{L}, -\Delta]\) is compact, there exists \( \delta_1 > 0 \) such that for any \( x, y \in [-\mathcal{L}, -\Delta]\)
\[
|x - y| < \delta_1 \quad \Rightarrow \quad |\phi(x) - \phi(y)| < {\epsilon}.
\]
Define:
\[
x_1 := \frac{x + \widehat{\alpha\tau} \cdot \text{sign}(x)}{\hat{\tau}}, \quad
x_2 := \frac{x + \alpha\tau \cdot \text{sign}(x)}{\tau}.
\]
Then
\[
\left| {\hat{q}_0(x)} - q_0(x) \right|
= \left| \frac{1}{\hat{\tau}} \cdot {\phi(x_1)} - \frac{1}{\tau}{\phi(x_2)} \right| 
= \left|  \frac{1}{\hat{\tau}} \cdot {\phi(x_1)} - \frac{1}{\tau}{\phi(x_2)} - \frac{1}{\tau}{\phi(x_1)} + \frac{1}{\tau}{\phi(x_1)}\right|.
\]
Using the triangle inequality, 
\begin{align*}
\left| {\hat{q}_0(x)} - q_0(x) \right|
&\leq \left|  \frac{1}{\hat{\tau}} \cdot {\phi(x_1)} - \frac{1}{\tau}{\phi(x_1)} \right| + \left| \frac{1}{\tau}{\phi(x_1)} - \frac{1}{\tau}{\phi(x_2)}\right| \\[2ex]
& \leq {\phi(x_1)}\left|  \frac{1}{\hat{\tau}}   - \frac{1}{\tau}  \right| + \frac{1}{\tau} \left|{\phi(x_1)} - {\phi(x_2)}\right|
\end{align*}
Since \( \phi \) is bounded on \( [-\mathcal{L}, -\Delta]\), there exists an upper bound \( \mathcal{M} > 0 \) such that \( \phi(x_1) \leq \mathcal{M} \), and since \( \hat{\tau} \xrightarrow{P} \tau \) and \( \widehat{\alpha\tau} \xrightarrow{P} \alpha\tau \),  \(\exists P_1 \  large \ enough \ s.t\  \forall p \geq P_1\) , with high probability:
\[
|x_1 - x_2| < \delta_1, \quad \left| {\hat{\tau}}-{\tau}  \right| < {\epsilon}.
\]So
\[
\left| {\phi(x_1) - \phi(x_2)} \right| \leq {\epsilon}.
\]
Therefore, with high probability, 
\[
\sup_{x \in  [-\mathcal{L}, -\Delta]} \left| {\hat{q}_0(x)} - {q_0(x)} \right|
\leq \mathcal{M} \cdot \frac{\epsilon}{\tau(\tau - \epsilon) } + \frac{\epsilon}{\tau}.
\]
\\
Similarly we can show that,\(\exists P_2\)  large enough s.t $\forall p \geq P_2$, with high probability
\\
$\underset{{x \in  [\Delta, \mathcal{L}]}}{\sup} \left| {\hat{q}_0(x)} - {q_0(x)} \right|
\leq \epsilon.
$
Now let \(P = max(P_1, P_2)\), then for all \(p\geq P\) and \(\forall \epsilon >0\) we have $\underset{{x \in I}}{\sup} \left| {\hat{q}_0(x)} - {q_0(x)} \right| \leq \epsilon$ with high probability. 
\\
We conclude that, as \( \epsilon \to 0 \),
\[
\underset{{x \in I}}{\sup} \left| {\hat{q}_0(x)} - {q_0(x)} \right| \xrightarrow{P} 0.
\]
\end{proof}

\begin{lemma}
\label{lem: lina: 2.4}
    \[
\frac{1}{p} \sum_{i=1}^{p} \mathbf{1}\left\{\hat{\beta}_i \neq 0, \frac{\hat{q}_0(\hat{\beta}_i)}{\hat{q}(\hat{\beta}_i)} \leq t\right\} \xrightarrow{P} 
\mathbb{P}\left(|\Pi + \tau Z| > \alpha\tau, \frac{q_0(\eta_{\alpha \tau}(\Pi + \tau Z))}{q(\eta_{\alpha \tau}(\Pi + \tau Z))} \leq t\right)
\]
\end{lemma}

\smallskip
\smallskip 

\begin{proof}[Proof of Lemma~\ref{lem: lina: 2.4}]
We need to show that $\forall \delta >0$, 
\begin{align*}
& \mathbb{P}(|\frac{1}{p} \sum_{i=1}^{p} \mathbf{1}\texttt{\{}{\hat{\beta_i} \neq 0,\frac{\hat q_0(\hat{\beta_i})}{\hat q(\hat{{\beta_i})}}\leq t}\texttt{\}} -\mathbb{P}(|\Pi + \tau Z| > \alpha\tau, \frac{q_0(\eta_{\alpha\tau}(\Pi + \tau Z))}{q(\eta_{\alpha\tau}(\Pi + \tau Z))} \leq t)|> \delta) \to 0.
\end{align*}
By the triangle inequality, 
\begin{align*}
&\mathbb{P}(|\frac{1}{p} \sum_{i=1}^{p} \mathbf{1}\texttt{\{}{\hat{\beta_i} \neq 0,\frac{\hat q_0(\hat{\beta_i})}{\hat q(\hat{{\beta_i})}}\leq t}\texttt{\}} -\mathbb{P}(|\Pi + \tau Z| > \alpha\tau, \frac{q_0(\eta_{\alpha\tau}(\Pi + \tau Z))}{q(\eta_{\alpha\tau}(\Pi + \tau Z))} \leq t)|> \delta)\\[2ex]
&\leq \mathbb{P}(|\frac{1}{p} \sum_{i=1}^{p} \mathbf{1}\texttt{\{}{\hat{\beta_i} \neq 0,\frac{\hat q_0(\hat{\beta_i})}{\hat q(\hat{{\beta_i})}}\leq t}\texttt{\}} -
\frac{1}{p}\sum_{i=1}^{p} 1\texttt{\{}{\hat{\beta_i} \neq 0,\frac{q_0({\hat{\beta_i}})}{q({{\hat{\beta_i}})}}\leq t}\texttt{\}}|> \frac{\delta}{2}) \\[2ex]
&\quad +  \mathbb{P}(|\frac{1}{p}\sum_{i=1}^{p} \mathbf{1}\texttt{\{}{\hat{\beta_i} \neq 0,\frac{q_0({\hat{\beta_i}})}{q({{\hat{\beta_i}})}}\leq t}\texttt{\}} - \mathbb{P}(|\Pi + \tau Z| > \alpha\tau, \frac{q_0(\eta_{\alpha\tau}(\Pi + \tau Z))}{q(\eta_{\alpha\tau}(\Pi + \tau Z))} \leq t)|> \frac{\delta}{2}) 
\end{align*}
\smallskip
Since in Theorem~\ref{thm:predictions-ol} we proved that 
\begin{align*}
&\frac{1}{p} \sum_{i=1}^{p} \mathbf{1}\texttt{\{}{\hat{\beta_i} \neq 0,\frac{q_0(\hat{\beta_i})}{q(\hat{{\beta_i})}}\leq t}\texttt{\}} \to \mathbb{P}(|\Pi + \tau Z| > \alpha\tau, \frac{q_0(\eta_{\alpha\tau}(\Pi + \tau Z))}{q(\eta_{\alpha\tau}(\Pi + \tau Z))} \leq t), 
\end{align*}
we have 
\begin{align*}
&\lim\limits_{p \to \infty}\mathbb{P}(|\frac{1}{p}\sum_{i=1}^{p} \mathbf{1}\texttt{\{}{\hat{\beta_i} \neq 0,\frac{q_0({\hat{\beta_i}})}{q({{\hat{\beta_i}})}}\leq t}\texttt{\}} - \mathbb{P}(|\Pi + \tau Z| > \alpha\tau, \frac{q_0(\eta_{\alpha\tau}(\Pi + \tau Z))}{q(\eta_{\alpha\tau}(\Pi + \tau Z))}
\leq t)|> \frac{\delta}{2}) = 0.
\end{align*}
The proof will be completed if we additionally show that
\begin{align*}
&\lim\limits_{p \to \infty}\mathbb{P}(|\frac{1}{p} \sum_{i=1}^{p} \mathbf{1}\texttt{\{}{\hat{\beta_i} \neq 0,\frac{\hat q_0(\hat{\beta_i})}{\hat q(\hat{{\beta_i})}}\leq t} \texttt{\}} -  
\frac{1}{p}\sum_{i=1}^{p} \mathbf{1} \texttt{\{}{\hat{\beta_i} \neq 0,\frac{q_0({\hat{\beta_i}})}{q({{\hat{\beta_i}})}}\leq t}\texttt{\}}|>\frac{\delta}{2})= 0. 
\end{align*}
Since we showed that $\frac{\hat q_0}{\hat q} \to \frac{ q_0}{q}$ uniformly in probability,
we get that for large p, the event  
\begin{align*}
A=\texttt{\{} \underset{1\leq i\leq p}{\sup} |\frac{\hat q_0(\hat{\beta_i})}{\hat q(\hat{{\beta_i})}}- \frac{q_0({\hat{\beta_i}})}{q({{\hat{\beta_i}})}}|\leq \delta' \texttt{\}} 
\end{align*}
occurs with high probability. 
Now, for large $p$, and on the event $A$, 
\smallskip
\smallskip

\noindent if $ \frac{q_0(\hat{\beta_i})}{q(\hat{\beta_i})} < t - \delta'$, then
\begin{align*}
\frac{\hat q_0(\hat{\beta_i})}{\hat q(\hat{{\beta_i})}} &= \frac{q_0(\hat{\beta_i})}{q(\hat{\beta_i})} + \frac{\hat q_0(\hat{\beta_i})}{\hat q(\hat{{\beta_i})}}-\frac{q_0(\hat{\beta_i})}{q(\hat{\beta_i})} \\[2ex]
&\leq \frac{q_0(\hat{\beta_i})}{q(\hat{\beta_i})} +| \frac{\hat q_0(\hat{\beta_i})}{\hat q(\hat{{\beta_i})}}-\frac{q_0(\hat{\beta_i})}{q(\hat{\beta_i})}| \\[2ex]
&\leq t-\delta'+ \delta'= t,
\end{align*}

\noindent As a result, on the event $A$, 
\smallskip
\begin{align} \label{eq: lemma 2.4 1}
&\mathbf{1} \texttt{\{} \hat{\beta_i} \neq 0,\frac{q_0(\hat{\beta_i})}{q(\hat{\beta_i})} < t - \delta' \texttt{\}} \leq \mathbf{1} \texttt{\{} \hat{\beta_i} \neq 0, \frac{ \hat q_0(\hat{\beta_i})}{\hat q(\hat{\beta_i})} < t \texttt{\}}. 
\end{align}

\smallskip

\noindent Similarly, if  $\frac{\hat q_0(\hat{\beta_i})}{\hat q(\hat{{\beta_i})}} < t ,$ then 
\begin{align*}
\frac{q_0(\hat{\beta_i})}{q(\hat{{\beta_i})}} &= \frac{\hat q_0(\hat{\beta_i})}{\hat q(\hat{\beta_i})} + \frac{q_0(\hat{\beta_i})}{ q(\hat{{\beta_i})}}-\frac{\hat q_0(\hat{\beta_i})}{\hat q(\hat{\beta_i})} \\[2ex]
&\leq \frac{\hat q_0(\hat{\beta_i})}{\hat q(\hat{\beta_i})} +|\frac{q_0(\hat{\beta_i})}{q(\hat{\beta_i})}-\frac{\hat q_0(\hat{\beta_i})}{\hat q(\hat{{\beta_i})}}|\\[2ex]
&\leq t  + \delta'.
\end{align*}

\noindent As a result, on the event $A$,  
\begin{align} \label{eq: Lemma 2.4 2}
\mathbf{1} \texttt{\{} \hat{\beta_i} \neq 0, \frac{\hat q_0(\hat{\beta_i})}{\hat q(\hat{\beta_i})} < t \texttt{\}} \leq \mathbf{1} \texttt{\{} \hat{\beta_i} \neq 0, \frac{ q_0(\hat{\beta_i})}{q(\hat{\beta_i})} < t + \delta' \texttt{\}}. 
\end{align}
\smallskip

\noindent Equations \eqref{eq: lemma 2.4 1} and \eqref{eq: Lemma 2.4 2} give us that for large $p$, with high probability, 
\smallskip
\smallskip
\begin{align*}
\mathbf{1} \texttt{\{} \hat{\beta_i} \neq 0, \frac{q_0(\hat{\beta_i})}{q(\hat{\beta_i})} < t - \delta'\texttt{\}}\leq \mathbf{1} \texttt{\{} \hat{\beta_i} \neq 0, \frac{\hat q_0(\hat{\beta_i})}{\hat q(\hat{\beta_i})} < t \texttt{\}} \leq \mathbf{1} \texttt{\{} \hat{\beta_i} \neq 0, \frac{ q_0(\hat{\beta_i})}{q(\hat{\beta_i})} < t +\delta'  \texttt{\}}, 
\end{align*}
\smallskip

\noindent which implies that 
\smallskip
\begin{align*}
\frac{1}{p} \sum_{i=1}^{p} \mathbf{1} \texttt{\{} \hat{\beta_i} \neq 0, \frac{q_0(\hat{\beta_i})}{q(\hat{\beta_i})} < t - \delta' \texttt{\}} &\leq  \frac{1}{p} \sum_{i=1}^{p} \mathbf{1} \texttt{\{} \hat{\beta_i} \neq 0, \frac{\hat q_0(\hat{\beta_i})}{\hat q(\hat{\beta_i})} < t \}
\leq \frac{1}{p} \sum_{i=1}^{p} \mathbf{1} \texttt{\{} \hat{\beta_i} \neq 0, \frac{ q_0(\hat{\beta_i})}{q(\hat{\beta_i})} < t +\delta'  \texttt{\}}
\end{align*}

\noindent If we subtract 
\noindent $\frac{1}{p} \sum_{i=1}^{p}  \mathbf{1} \texttt{\{} \hat{\beta_i} \neq 0, \frac{q_0(\hat{\beta_i})}{q(\hat{\beta_i})} <t\texttt{\}}$ from all sides we get, 
\begin{align*}
 &\frac{1}{p} \sum_{i=1}^{p} \mathbf{1}\texttt{\{} \hat{\beta_i} \neq 0,\frac{q_0(\hat{\beta_i})}{q(\hat{\beta_i})} < t - \delta' \texttt{\}} -\frac{1}{p} \sum_{i=1}^{p} \mathbf{1} \texttt{\{} \hat{\beta_i} \neq 0, \frac{q_0(\hat{\beta_i})}{q(\hat{\beta_i})} <t\texttt{\}} \\[2ex]
 &\leq \frac{1}{p} \sum_{i=1}^{p}\mathbf{1} \texttt{\{} \hat{\beta_i} \neq 0,\frac{\hat q_0(\hat{\beta_i})}{\hat q(\hat{\beta_i})} < t \texttt{\}} - \frac{1}{p} \sum_{i=1}^{p} \mathbf{1} \texttt{\{} \hat{\beta_i} \neq 0, \frac{q_0(\hat{\beta_i})}{q(\hat{\beta_i})} <t\texttt{\}}\\[2ex]
&\leq
\frac{1}{p} \sum_{i=1}^{p} \mathbf{1}\texttt{\{} \hat{\beta_i} \neq 0,\frac{ q_0(\hat{\beta_i})}{q(\hat{\beta_i})} < t +\delta'  \texttt{\}} - \frac{1}{p} \sum_{i=1}^{p} \mathbf{1} \texttt{\{} \hat{\beta_i} \neq 0,\frac{q_0(\hat{\beta_i})}{q(\hat{\beta_i})} <t\texttt{\}}
\end{align*}
\smallskip

\noindent Now let us look at 
\smallskip
\smallskip
\begin{align*}
&\frac{1}{p} \sum_{i=1}^{p} \mathbf{1} \texttt{\{} \hat{\beta_i} \neq 0, \frac{q_0(\hat{\beta_i})}{q(\hat{\beta_i})} < t - \delta' \texttt{\}} -\frac{1}{p} \sum_{i=1}^{p} \mathbf{1} \texttt{\{} \hat{\beta_i} \neq 0, \frac{q_0(\hat{\beta_i})}{q(\hat{\beta_i})} <t\texttt{\}}  \\[2ex]
&= -(\frac{1}{p} \sum_{i=1}^{p} \mathbf{1} \texttt{\{} \hat{\beta_i} \neq 0, t-\delta' <\frac{q_0(\hat{\beta_i})}{q(\hat{\beta_i})} <t\texttt{\}})
\end{align*}
\smallskip

\noindent By Theorem~\ref{thm:predictions-ol} we have
\begin{align*}
\frac{1}{p} \sum_{i=1}^{p} \mathbf{1} \texttt{\{} \hat{\beta_i} \neq 0, t-\delta' <\frac{q_0(\hat{\beta_i})}{q(\hat{\beta_i})} <t\texttt{\}} \to \mathbb{P}(|\Pi + \tau Z| > \alpha\tau, t-\delta' <
\frac{q_0(\eta_{\alpha\tau}(\Pi + \tau Z)) }{q(\eta_{\alpha\tau}(\Pi + \tau Z) )} <t) 
\end{align*}
So, 
\begin{align*}
    &\lim\limits_{\delta' \to 0} \lim\limits_{p \to \infty} \frac{1}{p} \sum_{i=1}^{p} \mathbf{1} \texttt{\{} \hat{\beta_i} \neq 0, t-\delta' <\frac{q_0(\hat{\beta_i})}{q(\hat{\beta_i})} <t\texttt{\}}\\[2ex]
    &=\lim\limits_{\delta' \to 0} \mathbb{P} (|\Pi + \tau Z| > \alpha\tau, t-\delta' <\frac{q_0(\eta_{\alpha\tau}(\Pi + \tau Z)) }{q(\eta_{\alpha\tau}(\Pi + \tau Z) )} <t).
    \end{align*}
\smallskip

\noindent The random variable 
$$
Y = \frac{q_0(\eta_{\alpha\tau}(\Pi + \tau Z)) }{q(\eta_{\alpha\tau}(\Pi + \tau Z) )}
$$ 
is continuous, with continuous p.d.f.~at $t$, so, if $\delta'$ is sufficiently small, then 
\[
\lim\limits_{\delta' \to 0} \mathbb{P}(|\Pi + \tau Z| > \alpha\tau,, t-\delta' <\frac{q_0(\eta_{\alpha\tau}(\Pi + \tau Z)) }{q(\eta_{\alpha\tau}(\Pi + \tau Z) )} <t)= 0, 
\]
and the same for $\frac{1}{p} \sum_{i=1}^{p} \mathbf{1} \texttt{\{} \hat{\beta_i} \neq 0, t <\frac{q_0(\hat{\beta_i})}{q(\hat{\beta_i})} <t + \delta' \texttt{\}}$. 
We therefore have 
\begin{align*}
&\lim\limits_{\delta' \to 0} \lim\limits_{p \to \infty} \frac{1}{p} \sum_{i=1}^{p} \mathbf{1} \texttt{\{} \hat{\beta_i} \neq 0,  t-\delta' <\frac{q_0(\hat{\beta_i})}{q(\hat{\beta_i})} <t\texttt{\}} \\[2ex] &=\lim\limits_{\delta' \to 0} \mathbb{P}(|\Pi + \tau Z| > \alpha\tau, t-\delta' <\frac{q_0(\eta_{\alpha\tau}(\Pi + \tau Z)) }{q(\eta_{\alpha\tau}(\Pi + \tau Z) )} <t) =0,
\end{align*}
and 
\smallskip
\begin{align*}
&\lim\limits_{\delta' \to 0} \lim\limits_{p \to \infty} \frac{1}{p} \sum_{i=1}^{p} \mathbf{1} \texttt{\{} \hat{\beta_i} \neq 0, t <\frac{q_0(\hat{\beta_i})}{q(\hat{\beta_i})} <t+\delta'\texttt{\}} \\[2ex]
&=\lim\limits_{\delta' \to 0} \mathbb{P}(|\Pi + \tau Z| > \alpha\tau, t<\frac{q_0(\eta_{\alpha\tau}(\Pi + \tau Z)) }{q(\eta_{\alpha\tau}(\Pi + \tau Z) )} <t+\delta') =0 
\end{align*}
\smallskip
So by the sandwich theorem we conclude that 
\[\frac{1}{p} \sum_{i=1}^{p}\mathbf{1} \texttt{\{} \hat{\beta_i} \neq 0, \frac{\hat q_0(\hat{\beta_i})}{\hat q(\hat{\beta_i})} < t \texttt{\}} - \frac{1}{p} \sum_{i=1}^{p} \mathbf{1} \texttt{\{} \hat{\beta_i} \neq 0, \frac{q_0(\hat{\beta_i})}{q(\hat{\beta_i})} <t\texttt{\}} \to 0.\] 

\end{proof}

\begin{lemma}
\label{lem: lina: 2.5}
    \[
\frac{1}{p} \sum_{i=1}^{p} \mathbf{1}\left\{\hat{\beta}_i \neq 0, \frac{\hat{q}_0(\hat{\beta}_i)}{\hat{q}(\hat{\beta}_i)} \leq t, \beta_i = 0\right\} \xrightarrow{P} 
\mathbb{P}\left(|\Pi + \tau Z| > \alpha\tau, \frac{q_0(\eta_{\alpha \tau}(\Pi + \tau Z))}{q(\eta_{\alpha \tau}(\Pi + \tau Z))} \leq t, \Pi = 0\right)
\]
\end{lemma}

\begin{proof}[Proof of Lemma~\ref{lem: lina: 2.5}]
Similar to the proof of \ref{lem: lina: 2.4}.

\end{proof}

\smallskip

\begin{proof}[Proof of Theorem~\ref{thm:predictions-eb}]
%
The statement \eqref{eq:tpp-oracle} follows from Lemma~\ref{lem: lina: 2.4} and Lemma~\ref{lem: lina: 2.5}, i.e 
\begin{align*}
    \operatorname{FDP}^{EB}(t; \lambda) \xrightarrow{P}     \frac{\mathbb{P}\left(|\Pi + \tau Z| > \alpha\tau, \frac{q_0(\eta_{\alpha \tau}(\Pi + \tau Z))}{q(\eta_{\alpha \tau}(\Pi + \tau Z))} \leq t, \Pi = 0\right)}{
\mathbb{P}\left(|\Pi + \tau Z| > \alpha\tau, \frac{q_0(\eta_{\alpha \tau}(\Pi + \tau Z))}{q(\eta_{\alpha \tau}(\Pi + \tau Z))} \leq t\right)} \equiv \operatorname{fdp}^*(t; \lambda)
\end{align*}
and 
\begin{align*}
\text{TPP}^{\text{EB}}(t; \lambda) &= \frac{\frac{1}{p} \sum_{i=1}^{p} \mathbf{1}\left\{\hat{\beta}_i \neq 0, \frac{\hat{q}_0(\hat{\beta}_i)}{\hat{q}(\hat{\beta}_i)} \leq t, \beta_i \neq 0\right\}}{\frac{1}{p} \sum_{i=1}^{p} \mathbf{1}(\beta_i \neq 0)}\\[2ex]
&=\frac{\frac{1}{p} \sum_{i=1}^{p} \mathbf{1}\left\{\hat{\beta}_i \neq 0, \frac{\hat{q}_0(\hat{\beta}_i)}{\hat{q}(\hat{\beta}_i)} \leq t\right\}-\frac{1}{p} \sum_{i=1}^{p} \mathbf{1}\left\{\hat{\beta}_i \neq 0, \frac{\hat{q}_0(\hat{\beta}_i)}{\hat{q}(\hat{\beta}_i)} \leq t, \beta_i = 0\right\}}{\frac{1}{p} \sum_{i=1}^{p} \mathbf{1}(\beta_i \neq 0)} \\[2ex]
 &\xrightarrow{P}\mathbb{P}(|\Pi + \tau Z| > \alpha\tau, \frac{q_0(\eta_{\alpha \tau}(\Pi + \tau Z))}{q(\eta_{\alpha \tau}(\Pi + \tau Z))} \leq t| \Pi \neq 0)\\
 &\equiv \operatorname{tpp}^*(t; \lambda).
 \end{align*}
\end{proof}

\begin{proof}[Proof of Theorem~\ref{thm:optimal-lambda}]
Define the random variables 
\begin{equation}
\label{eq:optimal-lambda-X}
X = \eta_{\alpha\tau}(\Xtil),\quad \quad \Xtil = \Pi + \tau Z,     
\end{equation}
where $\alpha,\tau$ are determined by the value of $\lambda\neq \lambda^*$ in the statement of the theorem, and consider the problem of testing 
\begin{equation}
\label{eq:optimal-lambda-H0}
H_0: \Pi = 0
\end{equation}
assuming only $X$ (not $\Xtil$) is observed. 
Note that $\Xtil$ cannot be recovered from $X$ in general (only if $X\neq 0$). 
In this univariate (Bayesian, since $\Pi$ is random) testing problem, let $\FPR = \PP(\lfdr(X)\leq t\lvert \Pi=0)$ and $\TPR = \PP(\lfdr(X)\leq t\lvert \Pi\neq0)$ be, respectively, the type I error probability and the power of the test that rejects when $\lfdr(X)\leq t$, where $\lfdr(\cdot)$ is the function in \eqref{lfdr}. 
Denote also $\FPR^*, \FPR^*$ the corresponding quantities for the test that rejects when $\lfdr^*(X)\leq t^*$. 
Throughout the proof, we refer to a test as a {\em level-$\xi$} test if it has power equal to $\xi$. 

We need to show that if $\TPR = \FPR^*=\xi$, then $\FPR^*\leq \FPR$. 
First, we claim that this would be true in the (hypothetical) problem where $\Xtil$ is observed instead of $X$:

\begin{proposition}
\label{prop:dominating}
Consider testing \eqref{eq:optimal-lambda-H0} in a setup where $\Xtil$ observed instead of $X$, and let $\widetilde{\FPR}$ and $\widetilde{\TPR}$ be the type I error probability and the power of the test that rejects \eqref{eq:optimal-lambda-H0} if $\widetilde{\lfdr}(\Xtil)\leq \tilde{t}$. 
If $\widetilde{\TPR}=\widetilde{\TPR}^*$, then $\widetilde{\FPR}^* \leq \widetilde{\FPR}$. 
\end{proposition}

To prove Proposition \ref{prop:dominating}, suppose for contradiction there exists $\lambda$ such that $\widetilde{\FPR} < \widetilde{\FPR}^*$. 
Define 
\begin{equation}
\label{eq:prop:dominating-Y}
\tilde{Y}^* := \tilde{X}^* + \sqrt{\tau^2-(\tau^*)^2} \tilde{Z}, 
    \end{equation}
where $\tilde{Z}\sim \calN(0,1)$ independently of everything, where we use the fact that $\tau^*< \tau$. 
Then since $(\tilde{Y}^*, \Pi)$ has the same distribution as $(\tilde{X}, \Pi)$, the assumption $\widetilde{\FPR} < \widetilde{\FPR}^*$ implies that there is a test based on \eqref{eq:prop:dominating-Y}, which is a legitimate (randomized) test based on $\Xtil^*$, that has smaller type I error probability than the level-$\xi$ likelihood ratio (LR) test in $\Xtil^*$, which rejects when $\widetilde{\lfdr}(\Xtil^*)\leq \tilde{t}^*$ (because $\widetilde{\lfdr}(\tilde{x})$ is monotone in the LR test based on $\Xtil$). 
This is a contradiction to the Neyman-Pearson lemma (to be exact, the version included in the proof of Theorem \ref{thm:phi-optimal}). 

\medskip
Returning to the proof of the main claim, since $t^*\leq t_0^*$, we have 
$$
\lfdr^*(X)\leq t^*\ \iff\ \widetilde{\lfdr}^*(\tilde{X})\leq t^*, 
$$
so the level-$\xi$ oracle test based on $X^*$ has the same type I error probability as the level-$\xi$ oracle test based on $\Xtil^*$. 
Now, we have shown that the latter has smaller type I error probability than the level-$\xi$ LR test based on $\Xtil$. 
On the other hand, the level-$\xi$ LR test based on $\Xtil$ is better (has smaller FPR) than any other (i.e., not LR) level-$\xi$ test based on $\Xtil$, in particular better than the level-$\xi$ LR test based on $X$ (since any test based on $X$ is also a test based on $\Xtil$, as $X=\eta_{\alpha\tau}(\Xtil)$ is a function of $X$). 
\end{proof}

\end{document}